\documentclass[11pt]{amsart}
\usepackage{amssymb
,amsthm
,amsmath
,amscd
,mathtools
,url
}
\usepackage[all]{xy}
\usepackage[top=35truemm,bottom=35truemm,left=30truemm,right=30truemm]{geometry}

\newcommand{\Z}{\mathbb{Z}}
\newcommand{\R}{\mathbb{R}}
\newcommand{\C}{\mathbb{C}}

\renewcommand{\AA}{\mathcal{A}}

\newcommand{\RR}{\mathcal{R}}
\newcommand{\Sc}{\mathcal{S}}

\newcommand{\GL}{\mathrm{GL}}
\newcommand{\SL}{\mathrm{SL}}
\newcommand{\SO}{\mathrm{SO}}
\newcommand{\Sp}{\mathrm{Sp}}

\newcommand{\id}{\mathrm{id}}
\newcommand{\Rep}{\mathrm{Rep}}
\newcommand{\Irr}{\mathrm{Irr}}
\newcommand{\unit}{\mathrm{unit}}
\newcommand{\gp}{\mathrm{gp}}
\newcommand{\temp}{\mathrm{temp}}
\newcommand{\Ind}{\mathrm{Ind}}
\newcommand{\Jac}{\mathrm{Jac}}
\newcommand{\Frob}{\mathrm{Frob}}
\newcommand{\semi}{\mathrm{s.s.}}
\newcommand{\Cent}{\mathrm{Cent}}
\newcommand{\im}{\mathrm{Im}}

\newcommand{\iif}{&\quad&\text{if }}
\newcommand{\other}{&\quad&\text{otherwise}}
\newcommand{\resp}{resp.~}
\renewcommand{\1}{\mathbf{1}}

\newcommand{\pair}[1]{\left\langle #1 \right\rangle}
\newcommand{\half}[1]{\frac{#1}{2}}
\newcommand{\ub}[1]{\underline{#1}}

\newtheorem{thm}{Theorem}[section]
\newtheorem{lem}[thm]{Lemma}
\newtheorem{prop}[thm]{Proposition}
\newtheorem{cor}[thm]{Corollary}
\newtheorem{rem}[thm]{Remark}
\newtheorem{defi}[thm]{Definition}
\newtheorem{ex}[thm]{Example}

\title{On an algorithm to compute derivatives}
\author{Hiraku Atobe}
\date{}
\address{
Department of Mathematics, Hokkaido University,
Kita 10, Nishi 8, Kita-Ku, Sapporo, Hokkaido, 060-0810, Japan 
}
\email{
atobe@math.sci.hokudai.ac.jp
}

\allowdisplaybreaks
\begin{document}
\maketitle

\begin{abstract}
In this paper, we complete Jantzen's algorithm to compute the highest derivatives 
of irreducible representations of $p$-adic odd special orthogonal groups or symplectic groups. 
As an application, 
we give some examples of the Langlands data of the Aubert duals 
of irreducible representations,
which are in the integral reducibility case. 
\end{abstract}

\tableofcontents

%\section{Introduction}
%\section{Introduction}
\section{Introduction}
The notion of \textit{derivatives} of admissible representations 
is an influential ingredient in representation theory of $p$-adic classical groups. 
Let $G$ be a split odd special orthogonal group $\SO_{2n+1}(F)$, or a symplectic group $\Sp_{2n}(F)$. 
Denote by $P_d=M_dN_d$ the standard parabolic subgroup of $G$ 
with Levi $M_d = \GL_d(F) \times G_0$ for some classical group of the same type as $G$. 
Fix an irreducible unitary supercuspidal representation $\rho$ of $\GL_d(F)$. 

\begin{defi}
Let $\pi$ be a smooth admissible representation of $G$ of finite length. 

\begin{enumerate}
\item
If the semisimplification of the Jacquet module $\Jac_{P_d}(\pi)$ along $P_d$ is of the form
\[
\semi \Jac_{P_d}(\pi) = \bigoplus_i \tau_i \boxtimes \pi_i, 
\]
we define the \textbf{partial Jacquet module} $\Jac_{\rho|\cdot|^x}(\pi)$ 
with respect to $\rho|\cdot|^x$ for $x \in \R$ by
\[
\Jac_{\rho|\cdot|^x}(\pi) = \bigoplus_{\substack{i \\ \tau_i \cong \rho|\cdot|^x}} \pi_i.
\]

\item
For a positive integer $k$, 
the \textbf{$k$-th derivative} $D_{\rho|\cdot|^x}^{(k)}(\pi)$ with respect to $\rho|\cdot|^x$
is defined by 
\[
D_{\rho|\cdot|^x}^{(k)}(\pi) = \frac{1}{k!} 
\underbrace{\Jac_{\rho|\cdot|^{x}} \circ \dots \circ \Jac_{\rho|\cdot|^{x}}}_k (\pi). 
\]
When $D_{\rho|\cdot|^x}^{(k)}(\pi) \not= 0$ but $D_{\rho|\cdot|^x}^{(k+1)}(\pi) = 0$, 
we say that $D_{\rho|\cdot|^x}^{(k)}(\pi)$ is the \textbf{highest derivative} (with respect to $\rho|\cdot|^x$). 

\end{enumerate}
\end{defi}
It is important that 
when $\rho|\cdot|^x$ is not self-dual, 
the highest derivative $D_{\rho|\cdot|^x}^{(k)}(\pi)$ of an irreducible representation is also irreducible, 
and $\pi$ is a unique irreducible subrepresentation of the parabolically induced representation 
$(\rho|\cdot|^x)^k \rtimes D_{\rho|\cdot|^x}^{(k)}(\pi)$ (see Proposition \ref{highest}). 
By these properties, the highest derivatives have many applications. 
For example: 
\begin{itemize}
\item
the proofs of the How duality conjecture by M{\'i}nguez \cite{M} and Gan--Takeda \cite{GT}; 

\item
another proof of the classification the unitary dual of general linear groups by Lapid--M{\'i}nguez \cite{LM1}; 

\item
several results on the irreducibility of parabolically induced representations 
by Jantzen \cite{J2} and Lapid--Tadi{\'c} \cite{LT}. 
\end{itemize}
\vskip 10pt

Jantzen \cite{J0} and M{\'i}nguez \cite{M1} 
obtained a complete description of the highest derivatives of irreducible representations of 
general linear groups $\GL_n(F)$ independently. 
It gives an algorithm to compute the Zelevinsky involutions. 
Similarly, if one were able to compute the highest derivatives of all irreducible representations, 
one can compute the Aubert dual of any irreducible representation
(see Theorem \ref{vs} below).
Jantzen \cite{J3} suggested an algorithm to compute the highest derivatives of irreducible representations. 
We will recall this algorithm in \S \ref{alg} below. 
According to this algorithm, 
the computation of the highest derivative of an arbitrary irreducible representation of classical group
is reduced to the ones of irreducible representations of the form 
$L((\rho|\cdot|^{-x})^a, \Delta_\rho[x-1, -x]^b; T)$
which is a unique irreducible (Langlands) subrepresentation of the standard module 
$(\rho|\cdot|^{-x})^a \times \Delta_\rho[x-1, -x]^b \rtimes T$, 
where 
$\rho$ is an irreducible unitary supercuspidal representation of $\GL_d(F)$, 
$x$ is a positive half-integer, 
$\Delta_\rho[x-1,-x]$ is a Steinberg representation of $\GL_{2dx}(F)$, 
and $T$ is an irreducible tempered representation of a small classical group. 
For these notations, see \S \ref{rep.cl} below. 
\vskip 10pt

For the problem to determine the highest derivative of $L((\rho|\cdot|^{-x})^a, \Delta_\rho[x-1, -x]^b; T)$, 
Jantzen gave an explicit formula for $x=1/2$ (\cite[Theorem 3.3]{J3}). 
Also, he suggested a strategy for $x > 1$ (\cite[\S 3.4, Cases 2, 3]{J3}). 
This strategy is an induction on $x$, i.e., 
the computation for $x$ is reduced to the one for $x-1$. 
Hence it would be possible to solve this problem when $x \in (1/2)\Z \setminus \Z$. 
Using this strategy, he computed some examples of the Aubert duals of certain irreducible representations
in the half-integral reducibility case (\cite[\S 4]{J3}). 
However, since \cite[Theorem 3.3]{J3} is already complicated (there are 5 cases), 
it is very hard to proceed with this induction. 
\vskip 10pt

In this paper, 
we give an algorithm (Theorem \ref{main}) to compute
the highest derivative of $L((\rho|\cdot|^{-x})^a, \Delta_\rho[x-1, -x]^b; T)$. 
One can also write down an explicit formula (Corollary \ref{ab}). 
This corollary together with Corollary \ref{k'} completes Jantzen's algorithm. 
When $a=0$, one might prove Corollary \ref{ab} by a similar argument to \cite[Theorem 3.3]{J3}, 
but we will give another argument. 
A new ingredient for the proof is two results of Xu on $A$-packets \cite{X2, X3}. 
The first is an estimation when derivatives of unitary representations of \textit{Arthur type} are nonzero
(Lemma \ref{xu}).
Namely, for a specific tuple $(a,b,T)$, we will find a ``good'' $A$-parameter $\psi$
such that $L((\rho|\cdot|^{-x})^a, \Delta_\rho[x-1, -x]^b; T)$ belongs to the $A$-packet $\Pi_\psi$
(e.g., see Example \ref{ex.xu1} and Proposition \ref{ex.xu2}). 
To compute the highest derivative, we will use M{\oe}glin's construction of $A$-packets
together with Xu's combinational result \cite[Theorem 6.1]{X3}. 
\vskip 10pt

This paper is organized as follows. 
In \S \ref{rep.cl}, 
we review several results in representation theory for classical groups. 
In particular, we explain how the highest derivatives of an irreducible representation $\pi$ 
determine its Langlands data almost completely (Theorem \ref{vs}). 
Also we recall Jantzen's algorithm to compute the highest derivatives in \S \ref{alg}. 
In \S \ref{packets}, we review Arthur's theory including Xu's lemma (Lemma \ref{xu})
and M{\oe}glin's construction (\S \ref{s.moe}).
In addition, some results on the irreducibility of parabolically induced representations are written in \S \ref{s.irr}. 
In \S \ref{s.main}, we state the main results (Theorem \ref{main} and Corollaries \ref{ab}, \ref{k'}).
Also in \S \ref{s.ex}, we give some examples of Aubert duals, which are in the integral case. 
Finally, we prove Theorem \ref{main} in \S \ref{s.proof}. 

%\subsection*{Acknowledgement}
\subsection*{Acknowledgement}
The author is grateful to Professor Alberto M{\'i}nguez
for telling the notion of derivatives and several results. 
He wishes to thank Professor Chris Jantzen to allow him to work on this topic. 
%Thanks are also due to the referee for the careful readings and the helpful comments.

%\subsection*{Notation}
\subsection*{Notation}
Let $F$ be a non-archimedean local field of characteristic zero. 
We denote by $W_F$ the Weil group of $F$.
The norm map $|\cdot| \colon W_F \rightarrow \R^\times$ is normalized so that $|\Frob| = q^{-1}$, 
where $\Frob \in W_F$ is a fixed (geometric) Frobenius element, 
and $q$ is the cardinality of the residual field of $F$.
\par

Each irreducible unitary supercuspidal representation $\rho$ of $\GL_d(F)$
is identified with the irreducible bounded representation of $W_F$ of dimension $d$ 
via the local Langlands correspondence for $\GL_d(F)$.
Through this paper, we fix such a $\rho$.
For each positive integer $a$, the unique irreducible algebraic representation of $\SL_2(\C)$
of dimension $a$ is denoted by $S_a$.
\par

For a $p$-adic group $G$, we denote by $\Rep(G)$ (\resp $\Irr(G)$) 
the set of equivalence classes of smooth admissible (\resp irreducible) representations of $G$.
For $\Pi \in \Rep(G)$, we write $\semi(\Pi)$ for the semisimplification of $\Pi$.

%\section{Representations of classical groups}\label{rep.cl}
%\section{Representations of classical groups}\label{rep.cl}
\section{Representations of classical groups}\label{rep.cl}
In this section, we recall some results on parabolically induced representations and Jacquet modules. 

%\subsection{Representations of $\GL_n(F)$}
\subsection{Representations of $\GL_n(F)$}\label{GL}
Let $P=MN$ be a standard parabolic subgroup of $\GL_n(F)$, 
i.e., $P$ contains the Borel subgroup consisting of upper half triangular matrices.
Then the Levi subgroup $M$ is isomorphic to $\GL_{n_1}(F) \times \dots \times \GL_{n_r}(F)$
with $n_1 + \dots + n_r = n$.
For smooth representations $\tau_1, \dots, \tau_r$ of $\GL_{n_1}(F), \dots, \GL_{n_r}(F)$, respectively, 
we denote the normalized parabolically induced representation by
\[
\tau_1 \times \dots \times \tau_r = 
\Ind_{P}^{\GL_n(F)}(\tau_1 \boxtimes \dots \boxtimes \tau_r).
\]
When $\tau_1 = \dots = \tau_r = \tau$, we write
\[
\tau^{r} = \underbrace{\tau \times \dots \times \tau}_r.
\]
\par

A \textbf{segment} is a symbol $[x,y]$, 
where $x,y \in \R$ with $x-y \in \Z$ and $x \geq y$.
We identify $[x,y]$ with the set $\{x, x-1, \dots, y\}$ so that $\#[x,y] = x-y+1$.
Let $\rho$ be an irreducible (unitary) supercuspidal representation of $\GL_d(F)$.
Then the normalized parabolically induced representation 
\[
\rho|\cdot|^{x} \times \dots \times \rho|\cdot|^y
\]
of $\GL_{d(x-y+1)}(F)$
has a unique irreducible subrepresentation, 
which is denoted by 
\[
\Delta_\rho[x,y]
\]
and 
is called a \textbf{Steinberg representation}. 
This is an essentially discrete series representation of $\GL_{d(x-y+1)}(F)$. 
\par

When $\tau_i = \Delta_{\rho_i}[x_i,y_i]$ with $x_1+y_1 \leq \dots \leq x_r+y_r$, 
the parabolically induced representation $\tau_1 \times \dots \times \tau_r$ is called
a \textbf{standard module}. 
The Langlands classification says that it has a unique irreducible subrepresentation, 
which is denoted by $L(\tau_1, \dots, \tau_r)$. 
This notation is used also for $\tau_1 \times \dots \times \tau_r$ which is isomorphic to a standard module. 
\par

When segments $[x_1,y_1], \dots, [x_t,y_t]$ satisfies that $x_i \geq y_i$, 
$x_1 < \dots < x_t$, $y_1 < \dots < y_t$ and $x_1 \equiv \dots \equiv x_t \bmod \Z$, 
we call $L(\Delta_\rho[x_1,y_1], \dots, \Delta_\rho[x_t,y_t])$ 
a \textbf{ladder representation}. 
This notion was introduced by Lapid--M{\'i}nguez \cite{LM}.
In particular, when $x_i = x_1+i-1$ and $y_i = y_1+i-1$ for $1 \leq i \leq t$, 
the ladder representation $L(\Delta_\rho[x_1,y_1], \dots, \Delta_\rho[x_t,y_t])$ 
is also called a \textbf{Speh representation}, which is a unitary representation.
\par

An irreducibility criterion for parabolically induced representations
of Steinberg representations is given by Zelevinsky. 
\begin{thm}[Zelevinsky {\cite[Theorem 9.7, Proposition 4.6]{Z}}]\label{zel}
Let $[x,y]$ and $[x',y']$ be segments, 
and let $\rho$ and $\rho'$ be irreducible unitary supercuspidal representations 
of $\GL_d(F)$ and $\GL_{d'}(F)$, respectively.
Then the parabolically induced representation 
\[
\Delta_\rho[x,y] \times \Delta_{\rho'}[x',y']
\]
is irreducible unless the following conditions hold: 
\begin{itemize}
\item
$\rho \cong \rho'$; 
\item
$[x, y] \not\subset [x', y']$ and $[x', y'] \not\subset [x, y]$ as sets; 
\item
$[x, y] \cup [x', y']$ is also a segment.
\end{itemize}
In this case, if $x+y < x'+y'$, then there exists an exact sequence
\[
\begin{CD}
0 
@>>> L(\Delta_\rho[x,y], \Delta_{\rho'}[x',y'])
@>>> \Delta_\rho[x,y] \times \Delta_{\rho'}[x',y']
@>>> \Delta_\rho[x',y] \times \Delta_\rho[x,y']
@>>> 0. 
\end{CD}
\]
Here, when $x = y'-1$, we omit $\Delta_\rho[x,y']$.
\end{thm}
For an irreducibility criterion for parabolically induced representations of Speh (\resp ladder) representations, 
see \cite{MW} (\resp \cite{LM1}).
\par

For a partition $(n_1,\dots, n_r)$ of $n$, 
we denote by $\Jac_{(n_1, \dots, n_r)}$ the normalized Jacquet functor on $\Rep(\GL_n(F))$
with respect to the standard parabolic subgroup $P = MN$ 
with $M \cong \GL_{n_1}(F) \times \dots \times \GL_{n_r}(F)$.
The Jacquet modules of ladder representations with respect to maximal parabolic subgroups 
are computed by Kret--Lapid \cite[Theorem 2.1]{KL} in general.
The following is a special case of this computation. 
\begin{prop}\label{JacGL}
Let $L(\Delta_\rho[x_1,y_1], \dots, \Delta_\rho[x_t,y_t])$ be a ladder representation of $\GL_{n}(F)$
so that $x_i \geq y_i$, $x_1 < \dots < x_t$, $y_1 < \dots < y_t$ and $x_1 \equiv \dots \equiv x_t \bmod \Z$. 
Then unless $n_1 \equiv 0 \bmod d$, we have
$\Jac_{(n_1,n-n_1)}(L(\Delta_\rho[x_1,y_1], \dots, \Delta_\rho[x_t,y_t])) = 0$.
Moreover: 
\begin{enumerate}
\item
The semisimplification of 
$\Jac_{(d,n-d)}(L(\Delta_\rho[x_1,y_1], \dots, \Delta_\rho[x_t,y_t]))$ is equal to
\[
\sum_{\substack{1 \leq i \leq t \\ x_{i-1} < x_i-1}}
\rho|\cdot|^{x_i} \boxtimes 
L(\Delta_\rho[x_1,y_1], \dots, \Delta_\rho[x_i-1,y_i], \dots, \Delta_\rho[x_t,y_t]).
\] 
Here, we omit the condition $x_{i-1} < x_i-1$ when $i=1$, 
and we remove $\Delta_\rho[x_i-1, y_i]$ when $x_i=y_i$. 
\item
The semisimplification of 
$\Jac_{(n-d,d)}(L(\Delta_\rho[x_1, y_1], \dots, \Delta_\rho[x_t, y_t]))$ is equal to
\[
\sum_{\substack{1 \leq i \leq t \\ y_i+1 < y_{i+1}}}
L(\Delta_\rho[x_1,y_1], \dots, \Delta_\rho[x_i,y_i+1], \dots, \Delta_\rho[x_t,y_t])
\boxtimes \rho|\cdot|^{y_i}.
\] 
Here, we omit the condition $y_i+1 < y_{i+1}$ when $i=t$, 
and we remove $\Delta_\rho[x_i, y_i+1]$ when $x_i=y_i$. 
\end{enumerate}
\end{prop}
\par

Let $\RR_n$ be the Grothendieck group of the category of smooth representations of $\GL_n(F)$ of finite length.
By the semisimplification, we identify the objects in this category with elements in $\RR_n$.
Elements in $\Irr(\GL_n(F))$ form a $\Z$-basis of $\RR_n$.
Set $\RR = \oplus_{n \geq 0} \RR_n$.
The parabolic induction functor gives a product 
\[
m \colon \RR \otimes \RR \rightarrow \RR, \ 
\tau_1 \otimes \tau_2 \mapsto \semi(\tau_1 \times \tau_2). 
\]
This product makes $\RR$ an associative commutative ring.
On the other hand, the Jacquet functor gives a coproduct
\begin{align*}
m^* \colon &\RR \rightarrow \RR \otimes \RR
\end{align*}
which is defined by the $\Z$-linear extension of 
\[
\Irr(\GL_n(F)) \ni \tau \mapsto \sum_{k = 0}^n \semi \Jac_{(k, n-k)}(\tau).
\]
Then $m$ and $m^*$ make $\RR$ a graded Hopf algebra, 
i.e., $m^* \colon \RR \rightarrow \RR \otimes \RR$ is a ring homomorphism.

\begin{defi}\label{derGL}
Let $\pi$ be an irreducible representation of $\GL_n(F)$.
\begin{enumerate}
\item
Suppose that
\[
\semi\Jac_{(d,n-d)}(\pi) = \bigoplus_{i \in I} \tau_i \boxtimes \pi_i, 
\quad
\semi\Jac_{(n-d,d)}(\pi) = \bigoplus_{j \in J} \pi_j' \boxtimes \tau'_j
\]
with $\tau_i, \tau_j'$ and $\pi_i,\pi_j'$ being irreducible representations of $\GL_d(F)$ and $\GL_{n-d}(F)$, respectively.
Then for $x \in \R$, we define the \textbf{left derivative} $L_{\rho|\cdot|^x}(\pi)$
and the \textbf{right derivative} $R_{\rho|\cdot|^x}(\pi)$ with respect to $\rho|\cdot|^x$ by 
\[
L_{\rho|\cdot|^x}(\pi) = \bigoplus_{\substack{i \in I \\ \tau_i \cong \rho|\cdot|^x}}\pi_i, 
\quad
R_{\rho|\cdot|^x}(\pi) = \bigoplus_{\substack{j \in J \\ \tau'_j \cong \rho|\cdot|^x}}\pi_j'.
\]

\item
For a positive integer $k$, 
we define the \textbf{$k$-th left and right derivatives} $L^{(k)}_{\rho|\cdot|^x}(\pi)$ and $R^{(k)}_{\rho|\cdot|^x}(\pi)$
with respect to $\rho|\cdot|^x$ by 
\[
L^{(k)}_{\rho|\cdot|^x}(\pi) = \frac{1}{k!}\underbrace{L_{\rho|\cdot|^x} \circ \dots \circ L_{\rho|\cdot|^x}}_k (\pi), 
\quad
R^{(k)}_{\rho|\cdot|^x}(\pi) = \frac{1}{k!}\underbrace{R_{\rho|\cdot|^x} \circ \dots \circ R_{\rho|\cdot|^x}}_k (\pi).
\]
We also set $L^{(0)}_{\rho|\cdot|^x}(\pi) = R^{(0)}_{\rho|\cdot|^x}(\pi) = \pi$. 

\item
When $L^{(k)}_{\rho|\cdot|^x}(\pi) \not=0$ but $L^{(k+1)}_{\rho|\cdot|^x}(\pi) = 0$, 
we call $L^{(k)}_{\rho|\cdot|^x}(\pi)$ the \textbf{highest left derivative}. 
We also define the \textbf{highest right derivative} similarly. 
\end{enumerate}
\end{defi}

By \cite[Lemma 2.1.2]{J0}, for any irreducible representation $\pi$, 
its highest derivatives $L^{(k)}_{\rho|\cdot|^x}(\pi)$ and $R^{(k')}_{\rho|\cdot|^x}(\pi)$ are irreducible. 
The following result was obtained by 
Jantzen \cite{J0} and M{\'i}nguez \cite[Th{\'e}or{\`e}me 7.5]{M1} independently. 
We adopt the statements of \cite[Propositions 2.1.4, 2.4.3, Theorems 2.2.1, 2.4.5]{J0}. 
For another reformulation, see \cite[Theorem 5.11]{LM1}. 
\par

\begin{thm}[Jantzen \cite{J0}, M{\'i}nguez \cite{M1}]\label{LR}
Let $\pi = L(\Delta_{\rho}[x_1,y_1], \dots, \Delta_{\rho}[x_r,y_r])$ be an irreducible representation. 
\begin{enumerate}
\item
We may assume that $y_1 \leq \dots \leq y_r$, and 
that if $y_j = y_{j+1}$, then $x_j \geq x_{j+1}$. 
For $1 \leq j \leq r$, define $n_x(j) = \#\{i \leq j \;|\; x_i = x\}$, and set $n_x(0) = 0$. 
Then with $k = \max_{j \geq 0}\{ n_x(j) - n_{x-1}(j)\}$, 
the left derivative $L_{\rho|\cdot|^x}^{(k)}(\pi)$ is highest. 
For $1 \leq m \leq k$, if we set $j_m = \min\{j \;|\; n_{x}(j)-n_{x-1}(j) = m\}$, 
then $L_{\rho|\cdot|}^{(k)}(\pi)$ is given from $\pi = L(\Delta_{\rho}[x_1,y_1], \dots, \Delta_{\rho}[x_r,y_r])$
by replacing $x_j = x$ with $x-1$ for all $j \in \{j_1, \dots, j_k\}$. 

\item
We may assume that $x_1 \leq \dots \leq x_r$, and 
that if $x_j = x_{j+1}$, then $y_j \geq y_{j+1}$. 
For $1 \leq j \leq r$, define $n'_y(j) = \#\{i \geq r-j+1 \;|\; y_i = y\}$, and set $n'_y(0) = 0$. 
Then with $k = \max_{j \geq 0}\{ n'_y(j) - n'_{y+1}(j)\}$, 
the right derivative $R_{\rho|\cdot|^y}^{(k)}(\pi)$ is highest. 
For $1 \leq m \leq k$, if we set $j_m = \min\{j \;|\; n'_{y}(j)-n'_{y+1}(j) = m\}$, 
then $R_{\rho|\cdot|^y}^{(k)}(\pi)$ is given from $\pi = L(\Delta_{\rho}[x_1,y_1], \dots, \Delta_{\rho}[x_r,y_r])$
by replacing $y_j = y$ with $y+1$ for all $j \in \{j_1, \dots, j_k\}$. 
\end{enumerate}
\end{thm}

%\subsection{Representations of $\SO_{2n+1}(F)$ and $\Sp_{2n}(F)$}
\subsection{Representations of $\SO_{2n+1}(F)$ and $\Sp_{2n}(F)$}\label{so-sp}
We set $G_n$ to be split $\SO_{2n+1}(F)$ or $\Sp_{2n}(F)$, 
i.e., $G_n$ is the group of $F$-points of the split algebraic group of type $B_n$ or $C_n$.
Fix a Borel subgroup of $G_n$, 
and let $P=MN$ be a standard parabolic subgroup of $G_n$.
Then the Levi part $M$ is of the form 
$\GL_{k_1}(F) \times \dots \times \GL_{k_r}(F) \times G_{n_0}$
such that $k_1+\dots+k_r+n_0=n$.
For a smooth representation $\tau_1 \boxtimes \dots \boxtimes \tau_r \boxtimes \pi_0$ of $M$, 
we denote the normalized parabolically induced representation by 
\[
\tau_1 \times \dots \times \tau_r \rtimes \pi_0 
= \Ind_P^{G_n}( \tau_1 \boxtimes \dots \boxtimes \tau_r \boxtimes \pi_0 ).
\]
The functor $\Ind_{P}^{G_n} \colon \Rep(M) \rightarrow \Rep(G_n)$ is exact. 
\par

On the other hand, for a smooth representation $\pi$ of $G_n$, 
we denote the normalized Jacquet module with respect to $P$ by 
\[
\Jac_P(\pi), 
\]
and its semisimplification by $\semi\Jac_P(\pi)$.
The functor $\Jac_{P} \colon \Rep(G_n) \rightarrow \Rep(M)$ is exact. 
The Frobenius reciprocity asserts that 
\[
\mathrm{Hom}_{G_n}(\pi, \Ind_P^{G_n}(\sigma)) \cong \mathrm{Hom}_{M}(\Jac_P(\pi), \sigma)
\]
for $\pi \in \Rep(G_n)$ and $\sigma \in \Rep(M)$.
\par

The maximal standard parabolic subgroup with Levi $\GL_{k}(F) \times G_{n-k}$ 
is denoted by $P_k = M_k N_k$ for $0 \leq k \leq n$.
Let $\RR(G_n)$ be the Grothendieck group of the category of 
smooth representations of $G_n$ of finite length. 
Set $\RR(G) = \oplus_{n \geq 0} \RR(G_n)$.
The parabolic induction defines a module structure
\[
\rtimes \colon \RR \otimes \RR(G) \rightarrow \RR(G), \ \tau \otimes \pi \mapsto \semi(\tau \rtimes \pi), 
\]
and the Jacquet functor defines a comodule structure
\[
\mu^* \colon \RR(G) \rightarrow \RR \otimes \RR(G)
\]
by
\[
\Irr(G_n) \ni \pi \mapsto \sum_{k=0}^{n} \semi\Jac_{P_k}(\pi).
\]
\par

Tadi{\'c} established a formula to compute $\mu^*$ for parabolically induced representations.
The contragredient functor $\tau \mapsto \tau^\vee$ defines an automorphism 
$\vee \colon \RR \rightarrow \RR$ in a natural way.
Let $s \colon \RR \otimes \RR \rightarrow \RR \otimes \RR$ be the homomorphism defined by
$\sum_i\tau_i \otimes \tau'_i \mapsto \sum_i \tau'_i \otimes \tau_i$.
\par

\begin{thm}[Tadi{\'c} \cite{T}]\label{tadic}
Consider the composition
\[
M^* = (m \otimes \id) \circ (\vee \otimes m^*) \circ s \circ m^* \colon \RR \rightarrow \RR \otimes \RR.
\]
Then for the maximal parabolic subgroup $P_k=M_kN_k$ of $G_n$
and for an admissible representation $\tau \boxtimes \pi$ of $M_k$, 
we have
\[
\mu^*(\tau \rtimes \pi) = M^*(\tau) \rtimes \mu^*(\pi).
\]
\end{thm}
\par

For any irreducible representation $\pi$ of $G_n$, 
there exists an irreducible representation $\tau_1 \boxtimes \dots \boxtimes \tau_r \boxtimes \sigma$ 
of a Levi subgroup $M = \GL_{k_1}(F) \times \dots \times \GL_{k_r}(F) \times G_{n_0}$ 
of some standard parabolic subgroup $P$
satisfying that
\begin{itemize}
\item
$\tau_i = \Delta_{\rho_i}[x_i, y_i]$ 
for some irreducible unitary supercuspidal representation $\rho_i$ of $\GL_{d_i}(F)$, 
and some segment $[x_i,y_i]$ with $x_i \geq y_i$; 
\item
$\sigma$ is an irreducible tempered representation of $G_{n_0}$; 
\item
$x_1+y_1 \leq \dots \leq x_r+y_r < 0$
\end{itemize}
such that $\pi$ is a unique irreducible subrepresentation of the parabolically induced representation 
$\tau_1 \times \dots \times \tau_r \rtimes \sigma$. 
In this case, we write
\[
\pi = L(\tau_1, \dots, \tau_r; \sigma), 
\]
and call it the \textbf{Langlands subrepresentation} of $\tau_1 \times \dots \times \tau_r \rtimes \sigma$. 
Note that $L(\tau_1, \dots, \tau_r; \sigma) \cong L(\tau'_1, \dots, \tau'_{r'}; \sigma')$
if and only if $\tau_1 \times \dots \times \tau_r \rtimes \sigma \cong \tau'_1 \times \dots \times \tau'_{r'} \rtimes \sigma'$. 
We refer $(\tau_1, \dots, \tau_r; \sigma)$ as the \textbf{Langlands data} of $\pi$.
For a detail, see \cite{K}. 

%\subsection{Derivatives}
\subsection{Derivatives}
\begin{defi}\label{der-classical}
Let $\pi$ be a smooth representation of $G_n$. 
\begin{enumerate}
\item
Consider $\semi\Jac_{P_d}(\pi)$ 
(and a fixed irreducible supercuspidal unitary representation $\rho$ of $\GL_d(F)$). 
If
\[
\semi\Jac_{P_d}(\pi) = \bigoplus_{i \in I} \tau_i \boxtimes \pi_i
\]
with $\tau_i$ (\resp $\pi_i$) being an irreducible representation of $\GL_d(F)$ (\resp $G_{n-d}$), 
for $x \in \R$, we define a \textbf{partial Jacquet module} $\Jac_{\rho|\cdot|^x}(\pi)$ by
\[
\Jac_{\rho|\cdot|^x}(\pi) = \bigoplus_{\substack{i \in I\\ \tau_i \cong \rho|\cdot|^x }}\pi_i.
\]
This is a representation of $G_{n-d}$.
For pairs $(\rho_1, x_1), \dots, (\rho_t, x_t)$, 
we also set 
$\Jac_{\rho_1|\cdot|^{x_1}, \dots, \rho_t|\cdot|^{x_t}} 
= \Jac_{\rho_t|\cdot|^{x_t}} \circ \dots \circ \Jac_{\rho_1|\cdot|^{x_1}}$.

\item
For a non-negative integer $k$, 
the \textbf{$k$-th derivative of $\pi$ with respect to $\rho|\cdot|^x$} is 
the $k$-th composition
\[
D_{\rho|\cdot|^x}^{(k)}(\pi) 
= \frac{1}{k!}\underbrace{\Jac_{\rho|\cdot|^x} \circ \dots \circ \Jac_{\rho|\cdot|^x}}_k(\pi).
\]
If $D_{\rho|\cdot|^x}^{(k)}(\pi) \not= 0$ but $D_{\rho|\cdot|^x}^{(k+1)}(\pi) = 0$, 
we call $D_{\rho|\cdot|^x}^{(k)}(\pi)$ \textbf{the highest derivative of $\pi$ with respect to $\rho|\cdot|^x$}.
\end{enumerate}
\end{defi}

The derivative $D_{\rho|\cdot|^x}^{(k)}(\pi)$ is a representation of some group $G_{n'}$ of the same type as $G_n$. 
In fact, it is characterized so that 
\[
\semi\Jac_{P_{dk}}(\pi) = 
(\rho|\cdot|^{x})^{k}
\otimes D_{\rho|\cdot|^x}^{(k)}(\pi) + \sum_{i \in I} \tau_i \otimes \pi_i
\]
where $\tau_i$ is an irreducible representation of $\GL_{dk}(F)$ 
such that $\tau_i \not\cong (\rho|\cdot|^{x})^{k}$. 
\par

The following is essentially the same as \cite[Lemma 3.1.3]{J1}.
For the convenience of the readers, we give a proof. 
\begin{prop}\label{highest}
Let $\pi$ be an irreducible representation of $G_n$, 
and $D_{\rho|\cdot|^x}^{(k)}(\pi)$ be the highest derivative. 

\begin{enumerate}
\item
There exists an irreducible representation $\pi'$ of some group $G_{n'}$ of the same type as $G_n$
such that $D_{\rho|\cdot|^x}^{(k)}(\pi) = m \cdot \pi'$ with a positive integer $m$.
Moreover, $\pi$ is an irreducible subrepresentation of the parabolically induced representation 
\[
\underbrace{\rho|\cdot|^x \times \dots \times \rho|\cdot|^x}_k \rtimes \pi'. 
\]

\item
If $\rho^\vee|\cdot|^{-x} \not\cong \rho|\cdot|^{x}$ (in particular, if $x \not= 0$), 
then $D_{\rho|\cdot|^x}^{(k)}(\pi)$ is irreducible, i.e., $m=1$. 
Moreover, in this case, 
$\pi$ is a unique irreducible subrepresentation of the above parabolically induced representation. 

\end{enumerate}
\end{prop}
\begin{proof}
By the same argument as the proof of \cite[Lemma 5.3]{X1}, 
there exists an irreducible representation $\pi'$ of some group $G_{n'}$ such that
\[
\pi \hookrightarrow 
\underbrace{\rho|\cdot|^x \times \dots \times \rho|\cdot|^x}_k \rtimes \pi'
= (\rho|\cdot|^{x})^{k} \rtimes \pi'.
\]
Note that $\Jac_{\rho|\cdot|^x}(\pi') = 0$ since $D_{\rho|\cdot|^x}^{(k)}(\pi)$ is the highest derivative. 
By Theorem \ref{tadic}, we see that
\[
\Jac_{\rho|\cdot|^x}\left(
(\rho|\cdot|^{x})^{k} \rtimes \pi'
\right)
= m_k \cdot (\rho|\cdot|^{x})^{k-1} \rtimes \pi', 
\]
where 
\[
m_k = 
\left\{
\begin{aligned}
&k  \iif \rho^\vee|\cdot|^{-x} \not\cong \rho|\cdot|^{x}, \\
&2k \iif \rho^\vee|\cdot|^{-x} \cong \rho|\cdot|^{x}.
\end{aligned}
\right.
\]
This implies that
\[
D_{\rho|\cdot|^x}^{(k)}\left(
(\rho|\cdot|^{x})^{k} \rtimes \pi'
\right)
= c_k \cdot \pi' 
\]
with
\[
c_k = 
\left\{
\begin{aligned}
&1  \iif \rho^\vee|\cdot|^{-x} \not\cong \rho|\cdot|^{x}, \\
&2^k \iif \rho^\vee|\cdot|^{-x} \cong \rho|\cdot|^{x}.
\end{aligned}
\right.
\]
Therefore $D_{\rho|\cdot|^x}^{(k)}(\pi) = m \cdot \pi'$ with a positive integer $m \leq c_k$.
This shows (1) and the first assertion of (2). 
For the last assertion of (2), 
we note that 
if $\pi_1$ is an irreducible subrepresentation of $(\rho|\cdot|^x)^k \rtimes \pi'$, 
then $D_{\rho|\cdot|^x}^{(k)}(\pi_1) \not= 0$. 
However, when $\rho^\vee|\cdot|^{-x} \not\cong \rho|\cdot|^{x}$, we have
\[
D_{\rho|\cdot|^x}^{(k)}\left(
(\rho|\cdot|^{x})^{k} \rtimes \pi' - \pi 
\right)=0. 
\]
This means that $(\rho|\cdot|^{x})^{k} \rtimes \pi'$ contains $\pi$
as a unique irreducible subrepresentation (with multiplicity one). 
\end{proof}

%\subsection{$\rho$-data}\label{s.data}
\subsection{$\rho$-data}\label{s.data}
In this subsection, 
we introduce the $\rho$-data of irreducible representations. 
\par

\begin{defi}\label{Mdata}
Let $\pi$ be an irreducible representation of $G_n$. 
For $\epsilon \in \{\pm\}$, 
the \textbf{$\rho$-data} of $\pi$ is of the form 
\[
M^\epsilon_\rho(\pi) = \left[ (x_1,k_1), \dots, (x_t,k_t); \pi_0 \right], 
\]
where $x_i$ is a real number and $k_i$ is a positive integer, 
defined inductively as follows. 

\begin{enumerate}
\item
If $\Jac_{\rho|\cdot|^x}(\pi) = 0$ for any $x \in \R$, 
we set $M^\epsilon_\rho(\pi) = \left[ \pi \right]$ (so that $t=0$ and $\pi_0 = \pi$).

\item
If $\Jac_{\rho|\cdot|^x}(\pi) \not= 0$ for some $x \in \R$, 
we set 
\[
x_1 = \left\{
\begin{aligned}
&\max\{x \in \R \;|\; \Jac_{\rho|\cdot|^x}(\pi) \not= 0\} \iif \epsilon = +, \\
&\min\{x \in \R \;|\; \Jac_{\rho|\cdot|^x}(\pi) \not= 0\} \iif \epsilon = -
\end{aligned}
\right.
\] 
and $k_1 \geq 1$ to be such that 
$D_{\rho|\cdot|^{x_1}}^{(k_1)}(\pi)$ is the highest derivative of $\pi$ with respect to $\rho|\cdot|^{x_1}$. 
By Proposition \ref{highest}, 
we can write $D_{\rho|\cdot|^{x_1}}^{(k_1)}(\pi) = m \cdot \pi'$ for some irreducible representation $\pi'$. 
Then we define 
\[
M^\epsilon_\rho(\pi) = \left[ (x_1,k_1); M^\epsilon_\rho(\pi') \right].
\]

\end{enumerate}
\end{defi}

Let $\pi$ be an irreducible representation of $G_n$.
Then one can define another irreducible representation $\hat\pi$, 
which is called the \textbf{Aubert dual} of $\pi$ (see \cite{Au}). 
It is known that 
\begin{itemize}
\item
$\hat{\hat\pi} = \pi$;
\item
if $\pi$ is supercuspidal, then $\hat\pi = \pi$; 
\item
$D_{\rho|\cdot|^{x}}^{(k)}(\hat\pi)$ is the Aubert dual of $D_{\rho^\vee|\cdot|^{-x}}^{(k)}(\pi)$. 
\end{itemize}
In particular, for $\epsilon \in \{\pm\}$, 
if 
\[
M^\epsilon_\rho(\pi) = \left[ (x_1,k_1), \dots, (x_t,k_t); \pi_0 \right] , 
\]
then 
\[
M^{-\epsilon}_{\rho^\vee}(\hat\pi) = \left[ (-x_1,k_1), \dots, (-x_t,k_t); \hat\pi_0 \right]. 
\]
\par

We will use $M^-_\rho(\pi)$ mainly. 
By taking the Aubert dual, several properties of $M^-_\rho$ 
can be translated into ones of $M^+_{\rho^\vee}$. 
The following is the most important property of $M^-_\rho$. 

\begin{thm}\label{data}
Let $\pi$ be an irreducible representation of $G_n$. 
Then the $\rho$-data $M^-_\rho(\pi)$ can be rewritten as 
\[
M^-_\rho(\pi) = \left[
(x_1^{(1)}, k_1^{(1)}), \dots, (x_{t_1}^{(1)}, k_{t_1}^{(1)}), 
\dots, 
(x_1^{(r)}, k_1^{(r)}), \dots, (x_{t_r}^{(r)}, k_{t_r}^{(r)}); 
\pi_0
\right]
\]
such that
\begin{itemize}
\item
$x_{j+1}^{(i)} = x_{j}^{(i)}-1$ and $k_{j+1}^{(i)} \leq k_j^{(i)}$ for $1 \leq i \leq r$ and $1 \leq j \leq t_i-1$; 
\item
$x_{1}^{(1)} < x_{1}^{(2)} < \dots < x_1^{(r)}$. 
\end{itemize}
Moreover, if we set $\tau^{(i)}_j = \Delta_\rho[x_1^{(i)}, x_{j}^{(i)}]$, 
then $\tau_j^{(i)} \times \tau_{j'}^{(i)} \cong \tau_{j'}^{(i)} \times \tau_{j}^{(i)}$ for any $1 \leq j,j' \leq t_r$, 
and $\pi$ is an irreducible subrepresentation of 
\[
\left(\bigtimes_{j=1}^{t_1} \left(\tau_{j}^{(1)}\right)^{k_j^{(1)}-k_{j+1}^{(1)}}\right)
\times \dots \times 
\left(\bigtimes_{j=1}^{t_r} \left(\tau_{j}^{(r)}\right)^{k_j^{(r)}-k_{j+1}^{(r)}}\right)
\rtimes \pi_0, 
\]
where we set $k_{t_i+1}^{(i)} = 0$ for $1 \leq i \leq r$.
\end{thm}
\begin{proof}
Write $M^-_\rho(\pi) = \left[ (x_1,k_1), \dots, (x_t,k_t); \pi_0 \right]$. 
Note that $x_{i+1} \not= x_i$ by definition.
Suppose that $x_{i+1} < x_i$. 
Replacing $\pi$ with the (unique) irreducible representation appearing in
$D_{\rho|\cdot|^{x_{i-1}}}^{(k_{i-1})} \circ \dots \circ D_{\rho|\cdot|^{x_1}}^{(k_1)}(\pi)$, 
we may assume that $i=1$.
If we write $D_{\rho|\cdot|^{x_2}}^{(k_2)} \circ D_{\rho|\cdot|^{x_1}}^{(k_1)}(\pi) = m \cdot \pi'$, 
then 
\[
\pi \hookrightarrow 
\underbrace{\rho|\cdot|^{x_1} \times \dots \times \rho|\cdot|^{x_1}}_{k_1} \times
\underbrace{\rho|\cdot|^{x_2} \times \dots \times \rho|\cdot|^{x_2}}_{k_2} \rtimes \pi'.
\]
If $x_2 < x_1$ but $x_2 \not= x_1 -1$, 
then $\rho|\cdot|^{x_1} \times \rho|\cdot|^{x_2} \cong \rho|\cdot|^{x_2} \times \rho|\cdot|^{x_1}$
so that $\Jac_{\rho|\cdot|^{x_2}}(\pi) \not= 0$.
This contradicts the definition of $x_1$. 
Hence $x_2 = x_1-1$. 
In this case, we have an exact sequence
\[
\begin{CD}
0 
@>>> \Delta_\rho[x_1,x_2] 
@>>> \rho|\cdot|^{x_1} \times \rho|\cdot|^{x_2} 
@>>> L(\rho|\cdot|^{x_2}, \rho|\cdot|^{x_1}) @>>> 0. 
\end{CD}
\]
Note that $\Jac_{\rho|\cdot|^{x_2}}(L(\rho|\cdot|^{x_2}, \rho|\cdot|^{x_1})) \not= 0$. 
Since 
$\rho|\cdot|^{x_1} \times L(\rho|\cdot|^{x_2}, \rho|\cdot|^{x_1}) 
\cong L(\rho|\cdot|^{x_2}, \rho|\cdot|^{x_1}) \times \rho|\cdot|^{x_1}$
by \cite[Theorem 4.2]{Z}, we must have
\[
\pi \hookrightarrow 
(\rho|\cdot|^{x_1})^{k_1-k_0} \times 
\Delta_\rho[x_1,x_2] ^{k_0} \times 
(\rho|\cdot|^{x_2})^{k_2-k_0} \rtimes \pi', 
\]
where we set $k_0 = \min\{k_1,k_2\}$. 
If $k_2 > k_1$ so that $k_0=k_1$, by Theorem \ref{zel}, 
we would have
\[
\pi \hookrightarrow 
(\rho|\cdot|^{x_2})^{k_2-k_1} \times 
\Delta_\rho[x_1,x_2] ^{k_0} \rtimes \pi', 
\]
which implies that $\Jac_{\rho|\cdot|^{x_2}}(\pi) \not= 0$.
This contradicts the definition of $x_1$. 
Hence we have $k_2 \leq k_1$.
\par

We conclude that if $x_1 > \dots > x_a$, then $x_{j+1} = x_j-1$ and $k_{j+1} \leq k_j$ for $1 \leq j \leq a-1$. 
Moreover, in this case, $\pi$ is a subrepresentation of 
\[
\Delta_\rho[x_1,x_1]^{k_1-k_2}
\times \dots \times 
\Delta_\rho[x_1,x_{a-1}]^{k_{a-1}-k_a}
\times 
\Delta_\rho[x_1,x_{a}]^{k_a}
\rtimes \pi'
\]
for some irreducible representation $\pi'$.
Now suppose that $x_{a} < x_{a+1}$. 
\begin{itemize}
\item
If $x_a < x_{a+1} < x_1$ and $x_{a+1} \not= x_b$ for any $1 \leq b < a$, 
by Theorem \ref{zel}, we have $\Jac_{\rho|\cdot|^{x_{a+1}}}(\pi) \not= 0$. 
This contradicts the definition of $x_1$. 

\item
If $x_{a+1} = x_b$ for some $1 \leq b < a$, by Theorem \ref{zel}, we have 
\[
D_{\rho|\cdot|^{x_{b}}}^{(k_{b}+1)} \circ 
D_{\rho|\cdot|^{x_{b-1}}}^{(k_{b-1})} \circ \dots \circ D_{\rho|\cdot|^{x_1}}^{(k_1)}(\pi) \not= 0.
\] 
This contradicts the definition of $k_b$. 
\end{itemize}
Hence we must have $x_{a+1} > x_1$.
Therefore $M^-_\rho(\pi)$ can be written as in the statement. 
The above argument also shows the other statements. 
\end{proof}

To obtain some consequences, let us prepare a useful lemma. 
\begin{lem}\label{ind}
Let $\pi$ be an irreducible representation of $G_n$. 
Write $M^-_\rho(\pi) = \left[ (x_1,k_1), \dots, (x_t,k_t); \pi_0 \right]$. 
Suppose that there exists an inclusion
\[
\pi \hookrightarrow \Delta_\rho[x,y] \rtimes \pi'
\]
with some irreducible representation $\pi'$ such that $x+y < 0$. 
Then there exists $i \leq t-(x-y)$ such that $x_i = x, x_{i+1} = x-1, \dots, x_{i+x-y} = y$. 
Moreover, if $i>1$, then $x_{i-1} < x$. 
\end{lem}
\begin{proof}
Write $M^-_\rho(\pi') = \left[ (x'_1,k'_1), \dots, (x'_{t'},k'_{t'}); \pi'_0 \right]$.
We prove the lemma by induction on $t'$. 
Note that $x_1 \leq x < -y$. 
If $t'=0$ or $x_1' \geq x$,  then $i=1$ and the assertion is trivial. 
Now we write $\pi \hookrightarrow (\rho|\cdot|^{x_1})^{k_1} \rtimes \pi_1$
and $\pi' \hookrightarrow (\rho|\cdot|^{x'_1})^{k_1'} \rtimes \pi'_1$
for some irreducible representations $\pi_1$ and $\pi_1'$. 
If $x_1' < x$ and $x_1' \not= y-1$, by Theorem \ref{zel}, we have
\[
\pi \hookrightarrow (\rho|\cdot|^{x'_1})^{k_1'} \times \Delta_\rho[x,y] \rtimes \pi'_1
\]
so that $x_1=x_1'$, $k_1=k_1'$ and $\pi_1 \hookrightarrow \Delta_\rho[x,y] \rtimes \pi'_1$.
By applying the induction hypothesis, we obtain the assertion. 
\par

Finally, we assume that $x_1' = y-1$. 
Since we have an exact sequence
\[
\begin{CD}
0 @>>> \Delta_\rho[x,y-1] @>>> \Delta_\rho[x,y] \times \rho|\cdot|^{y-1}
@>>> \rho|\cdot|^{y-1} \times \Delta_\rho[x,y], 
\end{CD}
\]
by Theorem \ref{zel}, we see that $\pi$ can be embedded into 
\begin{align*}
&(\rho|\cdot|^{y-1})^{k_1'} \times \Delta_\rho[x,y] \rtimes \pi'_1, 
\quad \text{or}\\
&(\rho|\cdot|^{y-1})^{k_1'-1} \times \Delta_\rho[x,y-1] \rtimes \pi'_1.
\end{align*}
In the former case, 
we have $x_1 = y-1$, $k_1 = k_1'$ and $\pi_1 \hookrightarrow \Delta_\rho[x,y] \rtimes \pi'_1$.
In the latter case, we have $(x_1,k_1) = (y-1,k_1'-1)$ or $k_1'=1$ and $x_1 \not= y-1$, 
and $\pi_1 \hookrightarrow \Delta_\rho[x,y-1] \rtimes \pi'_1$.
In both cases, the induction hypothesis gives the assertion. 
\end{proof}

\begin{cor}\label{minus}
Let $\pi$ be an irreducible representation of $G_n$. 
Write
\[
M^-_\rho(\pi) = \left[
(x_1^{(1)}, k_1^{(1)}), \dots, (x_{t_1}^{(1)}, k_{t_1}^{(1)}), 
\dots, 
(x_1^{(r)}, k_1^{(r)}), \dots, (x_{t_r}^{(r)}, k_{t_r}^{(r)}); 
\pi_0
\right]
\]
as in Theorem \ref{data}.
Suppose that $(x,y) = (x_{1}^{(i)}, x_{t_i}^{(i)})$ satisfies that 
$x+y < 0$ and $x+y \leq x_{1}^{(j)} + x_{t_j}^{(j)}$ for any $j <i$. 
Then there exists an irreducible representation $\pi'$ such that 
\[
\pi \hookrightarrow \Delta_\rho[x,y] \rtimes \pi'. 
\]
For any such $\pi'$, the $\rho$-data $M^-_\rho(\pi')$ is obtained from $M^-_\rho(\pi)$ 
by replacing $k_{1}^{(i)}, \dots, k_{t_i}^{(i)}$ with $k_{1}^{(i)}-1, \dots, k_{t_i}^{(i)}-1$, respectively. 
\end{cor}
\begin{proof}
By the assumption, 
we notice that $x < -y$ and $x > x_{1}^{(j)} \geq x_{t_j}^{(j)} > y$ for any $j < i$.
By Theorems \ref{data} and \ref{zel}, one can find an irreducible representation $\pi'$ such that 
\[
\pi \hookrightarrow \Delta_\rho[x,y] \rtimes \pi'. 
\]
We compute the $\rho$-data $M^-_\rho(\pi') = \left[ (x_1', k_1'), \dots, (x_t',k_t'); \pi_0' \right]$
by induction on $\sum_{i=1}^r t_i$ as in the proof of Lemma \ref{ind}.
\par

If $x_1' \geq x$, then $x_1^{(1)} = x$ so that $i=1$. 
In this case, the assertion is trivial. 
If $x_1' < x$ and $x_1' \not= y-1$, then $i>1$ and $(x_1',k_1') = (x_1^{(1)}, k_1^{(1)})$. 
Moreover, by Theorem \ref{zel}, we have
\[
D_{\rho|\cdot|^{x_1'}}^{(k_1')}(\pi) \hookrightarrow 
\Delta_\rho[x,y] \rtimes D_{\rho|\cdot|^{x_1'}}^{(k_1')}(\pi'). 
\]
By the induction hypothesis, we obtain the assertion. 
\par

To complete the proof, it suffices to show that $x_1'$ never equals to $y-1$. 
Suppose that $x_1' = y-1$. 
Then there exists an irreducible representation $\pi''$ 
such that $\pi' \hookrightarrow \rho|\cdot|^{y-1} \rtimes \pi''$ 
so that
\[
\pi \hookrightarrow \Delta_\rho[x,y] \times \rho|\cdot|^{y-1} \rtimes \pi''. 
\]
Since we have an exact sequence
\[
\begin{CD}
0 @>>> \Delta_\rho[x,y-1] @>>> \Delta_\rho[x,y] \times \rho|\cdot|^{y-1}
@>>> \rho|\cdot|^{y-1} \times \Delta_\rho[x,y], 
\end{CD}
\]
we have $\Jac_{\rho|\cdot|^{y-1}}(\pi) \not= 0$ or $\pi \hookrightarrow \Delta_\rho[x,y-1] \rtimes \pi''$.
Lemma \ref{ind} eliminates the latter case. 
In the former case, we must have $x_1^{(1)} \leq y-1 < x$. 
This implies that $i > 1$ and 
\[
x_1^{(1)}+x_{t_1}^{(1)} \leq 2 x_1^{(1)} \leq 2(y-1) < 2y \leq x+y, 
\]
which contradicts our assumption. 
This completes the proof. 
\end{proof}
\par

Also we can reformulate Casselman's tempered-ness criterion as follows. 

\begin{cor}\label{criterion}
Let $\pi$ be an irreducible representation of $G_n$. 
Then $\pi$ is tempered if and only if 
for any $\rho$, we can write
\[
M^-_\rho(\pi) = \left[
(x_1^{(1)}, k_1^{(1)}), \dots, (x_{t_1}^{(1)}, k_{t_1}^{(1)}), 
\dots, 
(x_1^{(r)}, k_1^{(r)}), \dots, (x_{t_r}^{(r)}, k_{t_r}^{(r)}); 
\pi_0
\right]
\]
as in Theorem \ref{data}
such that $x_{1}^{(i)}+x_{t_i}^{(i)} \geq 0$ for any $1 \leq i \leq r$.
\end{cor}
\begin{proof}
Suppose first that some $\rho$-data $M^-_\rho(\pi)$ of the above form 
has an index $i$ such that $x_{1}^{(i)}+x_{t_i}^{(i)} < 0$. 
We take $i$ so that $x_{1}^{(i)}+x_{t_i}^{(i)}$ achieves the minimum value. 
Then by Theorem \ref{data}, we find an irreducible representation $\pi'$ such that
\[
\pi \hookrightarrow \Delta_\rho[x_{1}^{(i)},x_{t_i}^{(i)}] \rtimes \pi', 
\]
or equivalently, 
\[
\Jac_{P_{dt_i}}(\pi) \twoheadrightarrow \Delta_\rho[x_{1}^{(i)},x_{t_i}^{(i)}] \otimes \pi'.
\]
By the Casselman criterion, we see that $\pi$ is not tempered. 
\par

Suppose conversely that $\pi$ is not tempered. 
By the Casselman criterion, 
there exist $\rho$, $[x,y]$ and $\pi'$ such that 
$\pi \hookrightarrow \Delta_\rho[x,y] \rtimes \pi'$ with $x+y < 0$.
By Lemma \ref{ind}, we conclude that there exists $i$ such that $x_1^{(i)} = x$ and $x_{t_i}^{(i)} \leq y$
so that $x_1^{(i)} + x_{t_i}^{(i)} \leq x+y < 0$. 
\end{proof}

Now we compare the $\rho$-data with the Langlands data. 
\begin{thm}\label{vs}
Let $\pi$ be an irreducible representation of $G_n$. 
Write
\[
M^-_\rho(\pi) = \left[
(x_1^{(1)}, k_1^{(1)}), \dots, (x_{t_1}^{(1)}, k_{t_1}^{(1)}), 
\dots, 
(x_1^{(r)}, k_1^{(r)}), \dots, (x_{t_r}^{(r)}, k_{t_r}^{(r)}); 
\pi_0
\right]
\]
and set $\tau^{(i)}_j = \Delta_\rho[x_1^{(i)}, x_{j}^{(i)}]$ as in Theorem \ref{data}. 
Suppose that $\pi = L(\tau_1, \dots, \tau_l; \sigma)$ with $\tau_i = \Delta_{\rho_i}[x_i, y_i]$. 
Then
\[
\bigtimes_{\substack{1 \leq i \leq l \\ \rho_i \cong \rho}} \tau_i
=
\bigtimes_{\substack{i,j \\ x_1^{(i)}+x_j^{(i)} < 0}} 
 \left(\tau^{(i)}_j\right)^{k_j^{(i)}-k_{j+1}^{(i)}}.
\]
Moreover, $\{M^-_{\rho}(\sigma)\}_\rho$ can be computed from $\{M^-_{\rho}(\pi)\}_\rho$ 
by Corollary \ref{minus}. 
\end{thm}
\begin{proof}
This follows from Theorem \ref{data} and Corollaries \ref{minus}, \ref{criterion}. 
\end{proof}
In fact, by Proposition \ref{ML} below,  
the tempered representation $\sigma$ is determined almost completely by $\{M^-_{\rho}(\sigma)\}_\rho$.
\par

Unfortunately, the map $\pi \mapsto M_\rho^\epsilon(\pi)$ is not injective. 
For example, 
when $\pi_0$ is supercuspidal and $\rho \rtimes \pi_0$ is reducible, 
this induced representation is semisimple of length two, i.e., 
$\rho \rtimes \pi_0 = \pi_1 \oplus \pi_2$ and $\pi_1 \not\cong \pi_2$. 
However $M^\epsilon_{\rho}(\pi_1) = M^\epsilon_{\rho}(\pi_2) = \left[ (0,1); \pi_0 \right]$ for $\epsilon \in \{\pm\}$. 
\par

%\subsection{Jantzen's algorithm}
\subsection{Jantzen's algorithm}\label{alg}
Let $\pi = L(\tau_1, \dots, \tau_r; \sigma)$ be an irreducible representation of $G_n$. 
Suppose that $\rho$ is self-dual and $x \in (1/2)\Z$.
We recall Jantzen's algorithm (\cite[\S 3.3]{J3}) 
to compute the highest derivative $D_{\rho|\cdot|^{x}}^{(k)}(\pi)$ with $x>0$. 

\begin{enumerate}
\item
Write 
$\tau_1 \times \dots \times \tau_r \cong \tau_1^{(1)} \times \dots \times \tau_{r_1}^{(1)} \times \Delta_\rho[x-1,-x]^b$
with $b$ maximal. 
Then
\[
\pi \hookrightarrow L(\tau_1^{(1)}, \dots, \tau_{r_1}^{(1)}) \times \Delta_\rho[x-1,-x]^b \rtimes \sigma.
\]

\item
Compute the right highest derivative 
$R_{\rho|\cdot|^{-x}}^{(a)}(L(\tau_1^{(1)}, \dots, \tau_{r_1}^{(1)})) 
= L(\tau_1^{(2)}, \dots, \tau_{r_2}^{(2)})$. 
Then Jantzen's Claim 1 says that
\[
\pi \hookrightarrow 
L(\tau_1^{(2)}, \dots, \tau_{r_2}^{(2)}) \rtimes L((\rho|\cdot|^{-x})^a, \Delta_\rho[x-1,-x]^b, \sigma).
\]

\item
Assume for a moment that we were able to compute the highest derivative 
$D_{\rho|\cdot|^x}^{(k_1)}(L((\rho|\cdot|^{-x})^a, \Delta_\rho[x-1,-x]^b, \sigma)) = \pi_1$. 
Then Jantzen's Claim 2 says that
\[
\pi \hookrightarrow 
L(\tau_1^{(2)}, \dots, \tau_{r_2}^{(2)}, (\rho|\cdot|^x)^{k_1}) \rtimes \pi_1.
\]

\item
Compute the left highest derivative 
$L_{\rho|\cdot|^{x}}^{(k)}(L(\tau_1^{(2)}, \dots, \tau_{r_2}^{(2)}, (\rho|\cdot|^x)^{k_1})) 
= L(\tau_1^{(3)}, \dots, \tau_{r_3}^{(3)})$. 
Then $D_{\rho|\cdot|^x}^{(k)}(\pi)$ is the highest derivative, and 
\[
D_{\rho|\cdot|^x}^{(k)}(\pi) \hookrightarrow L(\tau_1^{(3)}, \dots, \tau_{r_3}^{(3)}) \rtimes \pi_1.
\]

\item
By Theorem \ref{LR}, 
one can write 
$\tau_1^{(3)} \times \dots \times \tau_{r_3}^{(3)} 
\cong \tau_1^{(4)} \times \dots \times \tau_{r_4}^{(4)} \times (\rho|\cdot|^x)^{k_2}$
such that $\tau_i^{(4)}$ has a negative central exponent, and $k_2 \leq k_1$.

\item
There exists a unique irreducible representation $\pi_2$ 
of the form $L((\rho|\cdot|^{-x})^{a'}, \Delta_\rho[x-1,-x]^{b'}; \sigma')$
such that $D_{\rho|\cdot|^x}^{(k_2)}(\pi_2) = \pi_1$ is the highest derivative.
Jantzen's Claim 3 says that
\[
D_{\rho|\cdot|^x}^{(k)}(\pi) \hookrightarrow L(\tau_1^{(4)}, \dots, \tau_{r_4}^{(4)}) \rtimes \pi_2.
\]
Assume for a moment that we could specify $\pi_2$.

\item
There exists a unique irreducible representation $L(\tau_1^{(5)}, \dots, \tau_{r_5}^{(5)})$ such that
\[
R_{\rho|\cdot|^{-x}}^{(a')}(L(\tau_1^{(5)}, \dots, \tau_{r_5}^{(5)})) = L(\tau_1^{(4)}, \dots, \tau_{r_4}^{(4)})
\]
is the highest right derivative.
Moreover, $\tau_i^{(5)}$ has a negative central exponent. 
Then 
\[
D_{\rho|\cdot|^x}^{(k)}(\pi) = L(\tau_1^{(5)}, \dots, \tau_{r_5}^{(5)}, \Delta_\rho[x-1,-x]^{b'}; \sigma').
\]
\end{enumerate}

In conclusion, the computation of the highest derivative $D_{\rho|\cdot|^x}^{(k)}(\pi)$ 
is reduced to the one of $D_{\rho|\cdot|^x}^{(k)}(L((\rho|\cdot|^{-x})^a, \Delta_\rho[x-1,-x]^b, \sigma))$.
Jantzen's strategy of this computation is an induction on $x$. 
Namely, this computation for the case $x \geq 3/2$ can be reduced to the case for $x-1$.
Jantzen also compute the case $x=1/2$ (\cite[Theorem 3.3]{J3}). 
As a consequence, he gave some examples  
to compute the Langlands data for the Aubert duals $\hat\pi$ of certain irreducible representations $\pi$
in the half-integral reducibility case (\cite[\S 4]{J3}).
\par

In this paper, we will treat the general case.
To do this, a key idea is to use certain \textbf{Arthur packets}.

%\section{Arthur packets}
%\section{Arthur packets}
\section{Arthur packets}\label{packets}
In his book \cite{Ar}, for each $A$-parameter $\psi$, 
Arthur defined a finite (multi-)set $\Pi_\psi$ consisting of unitary representations of 
split $\SO_{2n+1}(F)$ or $\Sp_{2n}(F)$. 
In this section, we review his theory. 
\par

%\subsection{$A$-parameters}
\subsection{$A$-parameters}\label{Apara}
A homomorphism 
\[
\psi \colon W_F \times \SL_2(\C) \times \SL_2(\C) \rightarrow \GL_n(\C)
\]
is called an \textbf{$A$-parameter for $\GL_n(F)$} 
if 
\begin{itemize}
\item
$\psi(\Frob) \in \GL_n(\C)$ is semisimple and all its eigenvalues have absolute value $1$; 
\item
$\psi|W_F$ is smooth, i.e., has an open kernel; 
\item
$\psi|\SL_2(\C) \times \SL_2(\C)$ is algebraic.
\end{itemize}
The local Langlands correspondence for $\GL_n(F)$ asserts that 
there is a canonical bijection between
the set of irreducible unitary supercuspidal representations of $\GL_n(F)$
and 
the set of irreducible representations of $W_F$ of bounded images. 
We identify these two sets, and use the symbol $\rho$ for their elements. 
\par

Any irreducible representation of $W_F \times \SL_2(\C) \times \SL_2(\C)$ 
is of the form $\rho \boxtimes S_a \boxtimes S_b$, 
where $S_a$ is the unique irreducible representation of $\SL_2(\C)$ of dimension $a$.
We shortly write $\rho \boxtimes S_a = \rho \boxtimes S_a \boxtimes S_1$ 
and $\rho = \rho \boxtimes S_1 \boxtimes S_1$. 
For an $A$-parameter $\psi$, 
the multiplicity of $\rho \boxtimes S_a \boxtimes S_b$ in $\psi$
is denoted by $m_\psi(\rho \boxtimes S_a \boxtimes S_b)$.
When an $A$-parameter $\psi$ is decomposed into a direct sum
\[
\psi = \bigoplus_i \rho_i \boxtimes S_{a_i} \boxtimes S_{b_i}, 
\]
we define $\tau_\psi$ by the product of Speh representations
\[
\tau_\psi = \bigtimes_i 
L\left( 
\Delta_{\rho_i}\left[ \half{a_i-b_i}, -\half{a_i+b_i}+1 \right], \dots, \Delta_{\rho_i}\left[\half{a_i+b_i}-1, -\half{a_i-b_i} \right]
\right).
\]
\par

We say that an $A$-parameter $\psi \colon W_F \times \SL_2(\C) \times \SL_2(\C) \rightarrow \GL_k(\C)$ is
\textbf{symplectic} or \textbf{of symplectic type} (\resp \textbf{orthogonal} or \textbf{of orthogonal type})
if the image of $\psi$ is in $\Sp_k(\C)$ (so that $k$ is even) (\resp in $\mathrm{O}_k(\C)$). 
We call $\psi$ an \textbf{$A$-parameter for $\SO_{2n+1}(F)$} 
if it is an $A$-parameter for $\GL_{2n}(F)$ of symplectic type, i.e., 
\[
\psi \colon W_F \times \SL_2(\C) \times \SL_2(\C) \rightarrow \Sp_{2n}(\C). 
\]
Similarly, 
$\psi$ is called an \textbf{$A$-parameter for $\Sp_{2n}(F)$} 
if it is an $A$-parameter for $\GL_{2n+1}(F)$ of orthogonal type with the trivial determinant, i.e.,
\[
\psi \colon W_F \times \SL_2(\C) \times \SL_2(\C) \rightarrow \SO_{2n+1}(\C). 
\]
For $G_n = \SO_{2n+1}(F)$ (\resp $G_n = \Sp_{2n}(F)$), 
we let $\Psi(G_n)$ be the set of $\widehat{G_n}$-conjugacy classes of $A$-parameters for $G_n$, 
where $\widehat{G_n} = \Sp_{2n}(\C)$ (\resp $\widehat{G_n} = \SO_{2n+1}(\C)$).
We say that 
\begin{itemize}
\item
$\psi \in \Psi(G_n)$ is \textbf{tempered} 
if the restriction of $\psi$ to the second $\SL_2(\C)$ is trivial; 
\item
$\psi \in \Psi(G_n)$ is \textbf{of good parity} 
if $\psi$ is a sum of irreducible self-dual representations of the same type as $\psi$; 
\end{itemize}
We denote by $\Psi_\temp(G_n) = \Phi_\temp(G_n)$ (\resp $\Psi_\gp(G_n)$)
the subset of $\Psi(G)$ consisting of tempered $A$-parameters 
(\resp $A$-parameters of good parity).
Also, we put $\Phi_\gp(G_n) = \Phi_\temp(G_n) \cap \Psi_\gp(G_n)$. 
Set $\Psi_*(G) = \cup_{n \geq 0}\Psi_*(G_n)$ and $\Phi_*(G) = \cup_{n \geq 0}\Phi_*(G_n)$
for $* \in \{\emptyset, \temp. \gp\}$. 
\par

For $\psi \in \Psi(G_n)$, 
\textbf{the component group} is defined by $\Sc_\psi = \pi_0(\Cent_{\widehat{G_n}}(\im(\psi))/Z(\widehat{G_n}))$. 
This is an elementary two abelian group. 
It can be described as follows. 
Let $\psi \in \Psi(G)$. 
For simplicity, we assume that $\psi$ is of good parity. 
Hence we can decompose $\psi = \oplus_{i=1}^{t} \psi_i$, 
where $\psi_i$ is an irreducible representation (which is self-dual of the same type as $\psi$). 
We define an \textbf{enhanced component group} $\AA_\psi$ as 
\[
\AA_\psi = \bigoplus_{i=1}^t (\Z/2\Z)\alpha_{\psi_i}. 
\]
Namely, $\AA_\psi$ is a free $\Z/2\Z$-module of rank $t$
with a basis $\{\alpha_{\psi_i}\}$ associated with the irreducible components $\{\psi_i\}$. 
Then there exists a canonical surjection 
\[
\AA_\psi \twoheadrightarrow \Sc_\psi
\]
whose kernel is generated by the elements 
\begin{itemize}
\item
$z_\psi = \sum_{i=1}^t \alpha_{\psi_i}$; and
\item
$\alpha_{\psi_i} + \alpha_{\psi_{i'}}$ such that $\psi_i \cong \psi_{i'}$.  
\end{itemize}
Let $\widehat{\Sc_\psi}$ and $\widehat{\AA_\psi}$ be the Pontryagin duals of $\Sc_\psi$ and $\AA_\psi$, 
respectively. 
Via the surjection $\AA_\psi \twoheadrightarrow \Sc_\psi$, 
we may regard $\widehat{\Sc_\psi}$ as a subgroup of $\widehat{\AA_\psi}$. 
For $\eta \in \widehat{\AA_\psi}$, we write $\eta(\alpha_{\psi_i}) = \eta(\psi_i)$. 
By convention, we understand that $m_\psi(\rho \boxtimes S_0) = 1$ and $\eta(\rho \boxtimes S_0) = 1$.
\par

Let $\psi, \psi' \in \Psi_\gp(G)$. 
When there is a canonical inclusion $\AA_{\psi'} \hookrightarrow \AA_\psi$, 
for $\eta \in \widehat{\AA_\psi}$, we denote its restriction by $\eta' \in \widehat{\AA_{\psi'}}$.
For example, when $\psi' = \psi - \psi_0 + \psi_0'$ with $\psi_0$, $\psi_0'$ being irreducible, 
then we set $\eta'(\psi_0') = \eta(\psi_0)$.

\par

Let $\Irr_\unit(G_n)$ (\resp $\Irr_\temp(G_n)$) 
be the set of equivalence classes of irreducible unitary (\resp tempered) representations of $G_n$.
The local main theorem of Arthur's book is as follows. 
\begin{thm}[{\cite[Theorem 2.2.1, Proposition 7.4.1]{Ar}}]\label{arthur}
Let $G_n$ be a split $\SO_{2n+1}(F)$ or $\Sp_{2n}(F)$.
\begin{enumerate}
\item
For each $\psi \in \Psi(G_n)$, 
there is a finite multi-set $\Pi_\psi$ over $\Irr_\unit(G_n)$ with a map
\[
\Pi_\psi \rightarrow \widehat{\Sc_\psi},\; 
\pi \mapsto \pair{\cdot, \pi}_\psi
\]
satisfying certain (twisted and standard) endoscopic character identities. 
We call $\Pi_\psi$ the \textbf{$A$-packet} for $G_n$ associated with $\psi$. 

\item
When $\psi = \phi \in \Phi_\temp(G_n)$, 
the $A$-packet $\Pi_\phi$ is in fact a subset of $\Irr_\temp(G_n)$. 
Moreover, the map $\Pi_\phi \ni \pi \mapsto \pair{\cdot, \pi}_\phi \in \widehat{\Sc_\phi}$ is bijective, 
$\Pi_\phi \cap \Pi_{\phi'} = \emptyset$ for $\phi \not\cong \phi'$, 
and 
\[
\Irr_\temp(G_n) = \bigsqcup_{\phi \in \Phi_\temp(G_n)} \Pi_\phi. 
\]
When $\pi \in \Pi_\phi$ with $\eta = \pair{\cdot, \pi}_\phi \in \widehat{\Sc_\phi}$, 
we write $\pi = \pi(\phi, \eta)$.

\item
If $\psi = \oplus_{i}\rho_i \boxtimes S_{a_i} \boxtimes S_{b_i}$, 
set 
\[
\phi_{\psi, 0} = \bigoplus_{\substack{i \\ b_i \equiv 1 \bmod 2}} \rho_i \boxtimes S_{a_i}. 
\]
Then for any $\sigma \in \Pi_{\phi_{\psi,0}}$, 
the unique irreducible subrepresentation of 
\begin{align*}
&\left(
\bigtimes_{\substack{i \\ b_i \equiv 1 \bmod 2, b_i \not=1}} 
L\left(\Delta_{\rho_i}\left[\half{a_i-b_i}, -\half{a_i+b_i}+1\right], \dots, \Delta_{\rho_i}\left[\half{a_i-3}, -\half{a_i-1}\right]\right)
\right)
\\&\times 
\left(
\bigtimes_{\substack{i \\ b_i \equiv 0 \bmod 2}} 
L\left(\Delta_{\rho_i}\left[\half{a_i-b_i}, -\half{a_i+b_i}+1\right], \dots, \Delta_{\rho_i}\left[\half{a_i}-1, -\half{a_i}\right]\right)
\right)
\rtimes \sigma
\end{align*}
belongs to $\Pi_\psi$.
\end{enumerate}
\end{thm}

\begin{rem}
\begin{enumerate}
\item
The map $\Pi_\psi \ni \pi \mapsto \pair{\cdot, \pi}_\psi \in \widehat{\Sc_\psi}$ 
is not canonical when $G = \Sp_{2n}(F)$. 
To specify this, we implicitly fix an $F^{\times2}$-orbit of non-trivial additive characters of $F$ through this paper.

\item
In general, the map $\Pi_\psi \ni \pi \mapsto \pair{\cdot, \pi}_\psi \in \widehat{\Sc_\psi}$ 
is neither injective nor surjective. 

\item
In general, $\Pi_\psi$ can intersect with $\Pi_{\psi'}$ even when $\psi \not\cong \psi'$.
However, \cite[4.2 Corollaire]{MoeC} says that
if $\Pi_\psi \cap \Pi_{\psi'} \not= \emptyset$, then $\psi_d \cong \psi'_d$, 
where $\psi_d = \psi \circ \Delta$ is the diagonal restriction of $\psi$, 
i.e., $\Delta \colon W_F \times \SL_2(\C) \rightarrow W_F \times \SL_2(\C) \times \SL_2(\C)$ 
is defined by $\Delta(w,g) = (w,g,g)$. 

\end{enumerate}
\end{rem}

The following is a deep result of M{\oe}glin. 
\begin{thm}[M{\oe}glin \cite{Moe3}]\label{mult1}
The $A$-packet $\Pi_\psi$ is multiplicity-free, 
i.e., it is a subset of $\Irr_\unit(G)$. 
\end{thm}

Xu proved the following key lemma, whose proof uses the theory of endoscopy.
\begin{lem}[Xu {\cite[Proposition 8.3 (ii)]{X2}}]\label{xu}
Let $\psi = \oplus_{i \in I} \rho_i \boxtimes S_{a_i} \boxtimes S_{b_i} \in \Psi_\gp(G)$. 
For fixed $x \in \R$, if $D_{\rho|\cdot|^x}^{(k)}(\pi) \not= 0$ for some $\pi \in \Pi_\psi$, 
then
\[
k \leq \#\{i \in I \;|\; \rho_i \cong \rho,\; x = (a_i-b_i)/2\}.
\]
\end{lem}

The following is the first observation.
\begin{ex}\label{ex.xu1}
Let $\phi \in \Phi_\gp(G)$ and $\eta \in \widehat{\Sc_\phi}$. 
Fix $x \in (1/2)\Z$ with $x>0$. 
Suppose that $\rho \boxtimes S_{2x+1}$ is self-dual of the same type as $\phi$.
\begin{enumerate}
\item
Consider
\[
\pi = L(\Delta_\rho[x-1,-x]^b; \pi(\phi, \eta)).
\]
By Theorem \ref{arthur} (3), we have $\pi \in \Pi_\psi$ with 
\[
\psi = \phi + (\rho \boxtimes S_{2x} \boxtimes S_2)^b.
\]
In particular, by Lemma \ref{xu}, 
if $D_{\rho|\cdot|^x}^{(k)}(\pi) \not=0$, then $k \leq m_\phi(\rho \boxtimes S_{2x+1})$.

\item
Assume that $x=1$, and consider
\[
\pi = L((\rho|\cdot|^{-1})^a, \Delta_\rho[0,-1]^b; \pi(\phi, \eta)). 
\]
By Theorem \ref{arthur} (3), if $a \leq m_\phi(\rho)$, then $\pi \in \Pi_\psi$ with 
\[
\psi = \phi - \rho^a + (\rho \boxtimes S_1 \boxtimes S_3)^a + (\rho \boxtimes S_2 \boxtimes S_2)^b. 
\]
In particular, by Lemma \ref{xu}, 
if $D_{\rho|\cdot|^1}^{(k)}(\pi) \not=0$, then $k \leq m_\phi(\rho \boxtimes S_3)$.
\end{enumerate}
\end{ex}

%\subsection{Highest derivatives of tempered representations}
\subsection{Highest derivatives of tempered representations}
In \cite{J1, J2, J3}, Jantzen studied the derivatives of irreducible representations of $G_n$. 
To do this, he used the extended M{\oe}glin--Tadi{\'c} classification, 
which characterizes irreducible tempered representations 
by their cuspidal supports and the behavior of Jacquet modules. 
See \cite{J1}. 
Since the behavior of Jacquet modules of irreducible tempered representations 
are known by the previous paper \cite{At}, 
one can easily translate Jantzen's results in terms of the local Langlands correspondence 
(Theorem \ref{arthur} (2)).
\par

The highest derivatives of tempered representations are given as follows.

\begin{prop}[{\cite[Theorem 3.1]{J2}}]\label{error}
Let $\phi \in \Phi_\gp(G)$ and $\eta \in \widehat{\Sc_\phi}$. 
Fix a positive half-integer $x \in (1/2)\Z$. 
Suppose that $\rho \boxtimes S_{2x+1}$ is self-dual of the same type as $\phi$.
Denote $m = m_\phi(\rho \boxtimes S_{2x+1}) \geq 0$ by the multiplicity of $\rho \boxtimes S_{2x+1}$ in $\phi$. 

\begin{enumerate}
\item
When $x > 0$, we have 
\[
D_{\rho|\cdot|^x}^{(m)}(\pi(\phi, \eta)) = 
\pi\left(\phi - (\rho \boxtimes S_{2x+1})^{m} + (\rho \boxtimes S_{2x-1})^{m}, \eta \right). 
\]
It is the highest derivative if it is nonzero. 
Moreover, $D_{\rho|\cdot|^x}^{(m)}(\pi(\phi, \eta)) = 0$ if and only if 
\begin{itemize}
\item
$m > 0$; 
\item
$\rho \boxtimes S_{2x+1}$ is self-dual of the same type as $\phi$; 
\item
$\phi \supset \rho \boxtimes S_{2x-1}$ and  $\eta(\rho \boxtimes S_{2x+1}) \not= \eta(\rho \boxtimes S_{2x-1})$, 
\end{itemize}
where we understand that $\phi \supset \rho \boxtimes S_{0}$ and $\eta(\rho \boxtimes S_{0}) = 1$
when $\rho$ is self-dual of the opposite type to $\phi$. 

\item
Suppose that $m > 0$, $x > 0$ and $D_{\rho|\cdot|^x}^{(m)}(\pi(\phi, \eta)) = 0$. 
If $m$ is odd, then 
\[
D_{\rho|\cdot|^x}^{(m-1)}(\pi(\phi, \eta)) = 
\pi\left(\phi - (\rho \boxtimes S_{2x+1})^{m-1} + (\rho \boxtimes S_{2x-1})^{m-1}, \eta \right). 
\]
If $m$ is even, $D_{\rho|\cdot|^x}^{(m-1)}(\pi(\phi, \eta))$ is equal to
\[
L\left(
\Delta_\rho[x-1,x]; 
\pi\left(\phi -(\rho \boxtimes S_{2x+1})^{m} + (\rho \boxtimes S_{2x-1})^{m-2}, \eta\right) 
\right).
\]
In particular, in both cases, 
$D_{\rho|\cdot|^x}^{(m-1)}(\pi(\phi, \eta))$ is irreducible and the highest derivative. 
Moreover, $D_{\rho|\cdot|^x}^{(m-1)}(\pi(\phi, \eta))$ is tempered if and only if $m$ is odd.

\item
When $x = 0$, set $k = [m/2]$ to be the largest integer which is not greater than $m/2$. 
Then 
\[
D_\rho^{(k)}(\pi(\phi, \eta)) = c_k \cdot \pi\left(\phi -\rho^{2k}, \eta\right)
\]
with
\[
c_k = \left\{
\begin{aligned}
&2^{k-1} \iif \text{$\rho$ is of the same type as $\phi$ and $m$ is even}, \\
&2^k \other. 
\end{aligned}
\right.
\]
This is the highest derivative. 
\end{enumerate}
\end{prop}

\begin{rem}
\begin{enumerate}
\item
In \cite{At}, the highest derivatives of irreducible tempered representations was also considered, 
but Proposition 6.3 and Remark 6.4 in that paper are mistakes. 

\item
One can also prove Proposition \ref{error} by using results in \cite{At}. 

\item
When $\rho$ is not self-dual, $x \geq 0$ and $m = m_{\phi}(\rho \boxtimes S_{2x+1}) > 0$, 
we have $m_{\phi}(\rho^\vee \boxtimes S_{2x+1}) = m$ and 
\[
\pi(\phi, \eta) = \Delta_\rho[x,-x]^{m} \rtimes \pi(\phi_0, \eta_0), 
\]
where $\phi_0 = \phi - (\rho \oplus \rho^\vee)^{m} \boxtimes S_{2x+1}$, and $\eta_0 = \eta|\AA_{\phi_0}$. 
In this case, the highest derivative is 
\[
D_{\rho|\cdot|^x}^{(m)}(\pi(\phi, \eta)) 
= \Delta_\rho[x-1,x]^{m} \rtimes \pi(\phi_0, \eta_0).
\]
Here, when $x=0$, we omit $\Delta_\rho[x-1,x]$.
\end{enumerate}
\end{rem}

As a consequence, we have the following. 
\begin{prop}\label{ML}
Let $\pi = \pi(\phi, \eta)$ be an irreducible tempered representation of $G_n$. 
Fix $x \in (1/2)\Z$ and assume that $\rho \boxtimes S_{2|x|+1}$ is self-dual of the same type as $\phi$. 

\begin{enumerate}
\item
Suppose that $M^-_\rho(\pi) = \left[ (x_1, k_1), \dots, (x_t,k_t), (x,k); M^-_\rho(\pi_1) \right]$
where $x > 0$ and $\pi_1 = \pi(\phi_1, \eta_1)$ tempered ($t$ can be zero). 
Then $x=1/2$, or 
$\phi_1$ contains $\rho \boxtimes S_{2x-1}$ with multiplicity greater than or equal to $k$. 
Moreover, if $\phi_1$ contains $\rho \boxtimes S_{2x+1}$ 
and $\eta_1(\rho \boxtimes S_{2x+1}) \not= \eta_1(\rho \boxtimes S_{2x-1})$, 
then $k$ is even. 
Set 
\[
\phi_2 = \phi_1 - (\rho \boxtimes S_{2x-1})^{k} + (\rho \boxtimes S_{2x+1})^{k}, 
\]
and define $\eta_2 \in \widehat{\AA_{\phi_2}}$ so that 
$\eta_2(\rho' \boxtimes S_{a}) = \eta_1(\rho' \boxtimes S_{a})$ for any $(\rho', a) \not= (\rho, 2x+1)$ and 
\[
\eta_2(\rho \boxtimes S_{2x+1}) = 
\left\{
\begin{aligned}
&\eta_1(\rho \boxtimes S_{2x+1}) \iif \phi_1 \supset \rho \boxtimes S_{2x+1}, \\
&\eta_1(\rho \boxtimes S_{2x-1}) \other. 
\end{aligned}
\right. 
\]
Here, we understand that $\phi_1 \supset \rho \boxtimes S_0$ and $\eta_1(\rho \boxtimes S_0) = 1$ when $x=1/2$. 
Then 
\[
M^-_\rho(\pi) = \left[ (x_1, k_1), \dots, (x_t,k_t); M^-_\rho(\pi(\phi_2,\eta_2)) \right].
\]

\item
Suppose that $M^-_\rho(\pi) = \left[ (x_1, k_1), \dots, (x_t,k_t), (x,k); M^-_\rho(\pi_1) \right]$
where $x = 0$ and $\pi_1 = \pi(\phi_1, \eta_1)$ tempered ($t$ can be zero). 
Set 
\[
\phi_2 = \phi_1 + \rho^{2k}. 
\]
Then there exists $\eta_2 \in \widehat{\AA_{\phi_2}}$ with $\eta_2|\AA_{\phi_1} = \eta_1$
such that 
\[
M^-_\rho(\pi) = \left[ (x_1, k_1), \dots, (x_t,k_t); M^-_\rho(\pi(\phi_2,\eta_2)) \right]. 
\]

\item
Suppose that $M^-_\rho(\pi) = \left[ (x_1, k_1), \dots, (x_t,k_t), (-x,k); M^-_\rho(\pi_1) \right]$
where $x > 0$ and $\pi_1 = \pi(\phi_1, \eta_1)$ tempered. 
Then 
\begin{itemize}
\item
$k=1$; 
\item
there exists $1 \leq t' \leq t$ such that 
\[
M^-_\rho(\pi) = \left[ (x_1, k_1), \dots, (x_{t'},k_{t'}), (x,k'), (x-1,1), \dots, (-x,1); M^-_\rho(\pi_1) \right]
\]
for some odd $k'$; 

\item
$\phi_1$ contains $\rho \boxtimes S_{2x-1}$ with multiplicity greater than or equal to $k'$; 

\item
$\phi_1$ does not contain $\rho \boxtimes S_{2x+1}$.
\end{itemize}
Set 
\[
\phi_2 = \phi_1 + (\rho \boxtimes S_{2x+1})^{k'+1} - (\rho \boxtimes S_{2x-1})^{k'-1}, 
\]
and define $\eta_2 \in \widehat{\AA_{\phi_2}}$ so that 
$\eta_2(\rho' \boxtimes S_{a}) = \eta_1(\rho' \boxtimes S_{a})$ for any $(\rho', a) \not= (\rho, 2x+1)$ and 
$\eta_2(\rho \boxtimes S_{2x+1}) = -\eta_1(\rho \boxtimes S_{2x-1})$. 
Then 
\[
M^-_\rho(\pi) = \left[ (x_1, k_1), \dots, (x_{t'},k_{t'}); M^-_\rho(\pi(\phi_2,\eta_2)) \right]. 
\]
\end{enumerate}
\end{prop}
\begin{proof}
We show (1) and (2) so that $x \geq 0$. 
If an irreducible representation $\pi_2$ satisfies that 
$M^-_\rho(\pi) = \left[ (x_1, k_1), \dots, (x_t,k_t); M^-_\rho(\pi_2) \right]$, 
then $M^-_\rho(\pi_2) = \left[ (x, k); M^-_\rho(\pi_1) \right]$. 
By applying Corollary \ref{criterion} to $\pi_1$ and $\pi_2$, we see that $\pi_2$ is also tempered. 
If $\pi_2 = \pi(\phi_2,\eta_2)$, the relation between $(\phi_1, \eta_1)$ and $(\phi_2, \eta_2)$
is given in Proposition \ref{error}. 
Hence we obtain (1) and (2). 
\par

We show (3) so that $M^-_\rho(\pi) = \left[ (x_1, k_1), \dots, (x_t,k_t), (-x,k); M^-_\rho(\pi_1) \right]$ with $x>0$. 
By applying Corollary \ref{criterion} to $\pi$, 
we see that there exists $1 \leq t' \leq t$ such that 
$M^-_\rho(\pi) = \left[ (x_1, k_1), \dots, (x_{t'},k_{t'}), (x,k_1'), (x-1,k_2'), \dots, (-x,k_{2x+1}'); M^-_\rho(\pi_1) \right]$
with $k'_1 \geq \dots \geq k'_{2x+1} = k$. 
Take an irreducible representation $\pi_2$ such that 
$M^-_\rho(\pi) = \left[ (x_1, k_1), \dots, (x_{t'},k_{t'}); M^-_\rho(\pi_2) \right]$. 
Then $M^-_\rho(\pi_2) = \left[(x,k_1'), (x-1,k_2'), \dots, (-x,k_{2x+1}'); M^-_\rho(\pi_1) \right]$.
Since $\pi_1$ is tempered, by Corollary \ref{criterion}, we see that $\pi_2$ is also tempered. 
Write $\pi_2 = \pi(\phi_2,\eta_2)$. 
Then by Proposition \ref{error}, we have
\begin{itemize}
\item
$\phi_2$ contains $\rho \boxtimes S_{2x+1}$ with even multiplicity $2m > 0$; 
\item
$\phi_2$ contains $\rho \boxtimes S_{2x-1}$ and 
$\eta_2(\rho \boxtimes S_{2x+1}) \not= \eta_2(\rho \boxtimes S_{2x-1})$; 
\item
$k_1' = 2m-1$ and $k_2' = \dots = k'_{2x+1} = 1$; 
\item
$\phi_1 = \phi_2 - (\rho \boxtimes S_{2x+1})^{2m} + (\rho \boxtimes S_{2x-1})^{2m-2}$
so that $\phi_1$ contains $\rho \boxtimes S_{2x-1}$
with multiplicity greater than $2m-2 = k'_1-1$. 
\end{itemize}
Hence we obtain (3). 
\end{proof}

When $\rho \boxtimes S_{2|x|+1}$ is not self-dual of the same type as $\phi$, 
a similar (and easier) statement holds. 
We leave the detail for readers.
We note that: 
\begin{itemize}
\item
When $\rho \boxtimes S_{2|x|+1}$ is self-dual of the opposite type to $\phi$,
the case (3) cannot occur. 

\item
When $\rho$ is not self-dual, 
the case (1) cannot occur. 
\end{itemize}

%\subsection{Irreducibility of certain induced representations}
\subsection{Irreducibility of certain induced representations}\label{s.irr}
Using the highest derivatives, Jantzen proved some irreducibility of parabolically induced representations. 
For $\phi \in \Phi_\gp(G)$, 
we denote the multiplicity of $\rho \boxtimes S_a$ in $\phi$ by $m_\phi(\rho \boxtimes S_a)$.
For consistency, we define $m_\phi(\rho \boxtimes S_0) = 1$ and $\eta(\rho \boxtimes S_0) = +1$
if $\rho$ is self-dual of the opposite type to $\phi$.  

\begin{thm}[{\cite[Theorem 4.7]{J2}}]\label{irr0}
Let $\phi \in \Phi_\gp(G)$ and $\eta \in \widehat{\Sc_\phi}$. 
Fix $a \in \Z$ with $a \geq 2$ such that $\rho \boxtimes S_a$ is self-dual of the same type as $\phi$. 
Consider $\rho|\cdot|^{\half{a-1}} \rtimes \pi(\phi, \eta)$. 
\begin{enumerate}
\item
If $m_{\phi}(\rho \boxtimes S_{a-2}) = 0$, 
then $\rho|\cdot|^{\half{a-1}} \rtimes \pi(\phi, \eta)$ is irreducible. 
\item
If $m_{\phi}(\rho \boxtimes S_{a-2}) = 1$, then
\begin{enumerate}
\item
$\rho|\cdot|^{\half{a-1}} \rtimes \pi(\phi, \eta)$ is reducible 
if $m_\phi(\rho \boxtimes S_a) = 0$ or $\eta(\rho \boxtimes S_a) = \eta(\rho \boxtimes S_{a-2})$;
\item
$\rho|\cdot|^{\half{a-1}} \rtimes \pi(\phi, \eta)$ is irreducible 
if $m_\phi(\rho \boxtimes S_a) > 0$ and $\eta(\rho \boxtimes S_a) \not= \eta(\rho \boxtimes S_{a-2})$.
\end{enumerate}
\item
If $m_\phi(\rho \boxtimes S_{a-2}) \geq 2$, 
then $\rho|\cdot|^{\half{a-1}} \rtimes \pi(\phi, \eta)$ is reducible. 
\end{enumerate}
\end{thm}

We also need the irreducibility of other parabolically induced representations. 
\begin{prop}\label{irr01}
Let $\phi \in \Phi_\gp(G)$ and $\eta \in \widehat{\Sc_\phi}$. 
Fix $x \in (1/2)\Z$ with $x \geq 1$ such that $\rho \boxtimes S_{2x+1}$ is self-dual of the same type as $\phi$. 
\begin{enumerate}
\item
$\Delta_\rho[x-1,-x] \rtimes \pi(\phi, \eta)$ is irreducible 
if and only if $\phi$ contains both $\rho \boxtimes S_{2x+1}$ and $\rho \boxtimes S_{2x-1}$, 
and $\eta(\rho \boxtimes S_{2x+1}) \not= \eta(\rho \boxtimes S_{2x-1})$. 

\item
If $\phi$ contains both $\rho \boxtimes S_{2x+1}$ and $\rho \boxtimes S_{2x-1}$, 
but if $\eta(\rho \boxtimes S_{2x+1}) = \eta(\rho \boxtimes S_{2x-1})$, 
then $\Delta_\rho[x-1,-x] \rtimes L(\Delta_\rho[0,-1]; \pi(\phi, \eta))$ is irreducible. 
\end{enumerate}
\end{prop}
\begin{proof}
The only if part of (1) is proven in \cite[Proposition 5.2]{At}. 
The if part is similar to that proposition. 
Assume that $\phi$ contains both $\rho \boxtimes S_{2x+1}$ and $\rho \boxtimes S_{2x-1}$ 
and $\eta(\rho \boxtimes S_{2x+1}) \not= \eta(\rho \boxtimes S_{2x-1})$. 
Suppose that $\Delta_\rho[x-1,-x] \rtimes \pi(\phi, \eta)$ is reducible. 
Take an irreducible quotient $\pi'$ of $\Delta_\rho[x-1,-x] \rtimes \pi(\phi, \eta)$. 
Since the Langlands subrepresentation appears in the standard module 
as subquotients with multiplicity one, 
we see that $\pi'$ is tempered. 
Moreover, $\pi' \in \Pi_{\phi'}$ with $\phi' = \phi + \rho \boxtimes (S_{2x+1} + S_{2x-1})$.
We write $\pi' = \pi(\phi', \eta')$. 
When $D_{\rho|\cdot|^{x}}^{(k)}(\pi(\phi, \eta))$ is the highest derivative, 
by Proposition \ref{error}, we see that $D_{\rho|\cdot|^{x}}^{(k+1)}(\pi')$ is equal to
the irreducible induction $\Delta_\rho[x-1,-(x-1)] \rtimes D_{\rho|\cdot|^{x}}^{(k)}(\pi(\phi, \eta))$. 
If $\eta'(\rho \boxtimes S_{2x+1}) = \eta'(\rho \boxtimes S_{2x-1})$, then 
$D_{\rho|\cdot|^{x}}^{(k+1)}(\pi')$ is tempered 
such that its $L$-parameter $\phi''$ does not contain $\rho \boxtimes S_{2x+1}$, 
whereas, if $\Delta_\rho[x-1,-(x-1)] \rtimes D_{\rho|\cdot|^{x}}^{(k)}(\pi(\phi, \eta))$ is tempered, 
then its $L$-parameter must contain $\rho \boxtimes S_{2x+1}$. 
This is a contradiction, so that we have $\eta'(\rho \boxtimes S_{2x+1}) \not= \eta'(\rho \boxtimes S_{2x-1})$. In addition, we have equations of the highest derivatives 
\[
\left\{
\begin{aligned}
&D_{\rho|\cdot|^{x-1}}^{(l+2)} \circ D_{\rho|\cdot|^{x}}^{(k+1)}(\pi')
=
\Delta_\rho[x-2,-(x-2)] \rtimes D_{\rho|\cdot|^{x-1}}^{(l)} \circ D_{\rho|\cdot|^{x}}^{(k)}(\pi(\phi, \eta))
\iif x \geq2, \\
&D_{\rho|\cdot|^{\half{1}}}^{(l+2)} \circ D_{\rho|\cdot|^{\half{3}}}^{(k+1)}(\pi')
=
D_{\rho|\cdot|^{\half{1}}}^{(l)} \circ D_{\rho|\cdot|^{\half{3}}}^{(k)}(\pi(\phi, \eta))
\iif x =\half{3}, \\
&D_{\rho|\cdot|^{0}}^{(l+1)} \circ D_{\rho|\cdot|^{1}}^{(k+1)}(\pi')
= 2 \cdot
D_{\rho|\cdot|^{0}}^{(l)} \circ D_{\rho|\cdot|^{1}}^{(k)}(\pi(\phi, \eta))
\iif x =1.
\end{aligned}
\right.
\]
In each case, by Proposition \ref{error}, 
$\Delta_\rho[x-2, -x]$ appears in exactly one of the left or right hand side. 
This is a contradiction.
\par

(2) is proven similarly to \cite[Lemma 6.3]{J3}.
We give a sketch of the proof. 
\begin{itemize}
\item
For $\phi' = \phi - \rho \boxtimes (S_{2x+1} + S_{2x-1})$, 
we have
\[
\pi(\phi, \eta) \hookrightarrow \Delta_\rho[x,-(x-1)] \rtimes \pi(\phi', \eta').
\] 

\item
If we set $\pi = L(\Delta_\rho[x-1,-x]; \pi(\phi, \eta))$, 
then $\pi \hookrightarrow \Delta_{\rho}[x-1,-x] \times \Delta_\rho[x,-(x-1)] \rtimes \pi(\phi', \eta')$. 
This implies that $\pi \hookrightarrow L(\Delta_{\rho}[x-1,-x], \Delta_\rho[x,-(x-1)]) \rtimes \pi(\phi', \eta')$ or 
$\pi \hookrightarrow \Delta_{\rho}[x,-x] \times \Delta_\rho[x-1,-(x-1)] \rtimes \pi(\phi', \eta')$. 
Since $\pi$ is non-tempered, the former case must hold. 

\item
Put $m = m_{\phi}(\rho \boxtimes S_{2x+1})$. 
Since 
$L_{\rho|\cdot|^x}(L(\Delta_{\rho}[x-1,-x], \Delta_\rho[x,-(x-1)])) 
= R_{\rho|\cdot|^{-x}}(L(\Delta_{\rho}[x-1,-x], \Delta_\rho[x,-(x-1)])) = 0$, 
and since $L(\Delta_{\rho}[x-1,-x], \Delta_\rho[x,-(x-1)])$ commutes with $\rho|\cdot|^x$, 
we see that $D_{\rho|\cdot|^x}^{(m-1)}(\pi)$ is the highest derivative, 
and 
\[
D_{\rho|\cdot|^x}^{(m-1)}(\pi) \hookrightarrow 
L(\Delta_{\rho}[x-1,-x], \Delta_\rho[x,-(x-1)]) \rtimes D_{\rho|\cdot|^x}^{(m-1)}(\pi(\phi', \eta')). 
\]

\item
Set $\lambda = L(\Delta_\rho[x-1,-x]^2; \pi(\phi, \eta))$. 
Since up to semisimplification
\begin{align*}
\Delta_\rho[x-1,-x]^2 \rtimes \pi(\phi, \eta) 
=& 
\Delta_\rho[x-1,-x] \times \Delta_\rho[x,-(x-1)] \rtimes \pi(\phi, \eta) 
\\=&
L(\Delta_\rho[x-1,-x], \Delta_\rho[x,-(x-1)]) \rtimes \pi(\phi, \eta) 
\\&\oplus
\Delta_\rho[x,-x] \times \Delta_\rho[x-1,-(x-1)] \rtimes \pi(\phi, \eta), 
\end{align*}
we see that $\lambda$ is a subrepresentation of $L(\Delta_\rho[x-1,-x], \Delta_\rho[x,-(x-1)]) \rtimes \pi(\phi, \eta)$.
In particular, 
\[
D_{\rho|\cdot|^x}^{(m)}(\lambda) \hookrightarrow
L(\Delta_\rho[x-1,-x], \Delta_\rho[x,-(x-1)]) \rtimes D_{\rho|\cdot|^x}^{(m)}(\pi(\phi, \eta))
\]
is the highest derivative. 

\item
Since $\lambda \hookrightarrow \Delta_\rho[x-1,-x] \rtimes \pi$, 
we must have 
\[
D_{\rho|\cdot|^x}^{(m)}(\lambda) \subset \Delta_\rho[x-1,-(x-1)] \rtimes D_{\rho|\cdot|^x}^{(m-1)}(\pi).
\]
In particular, 
since $\lambda \hookrightarrow (\rho|\cdot|^x)^m \times \Delta_\rho[x-1,-(x-1)] \rtimes D_{\rho|\cdot|^x}^{(m-1)}(\pi)$, 
we have 
$\lambda \hookrightarrow \Delta_\rho[x,-(x-1)] \times (\rho|\cdot|^x)^{m-1} \rtimes D_{\rho|\cdot|^x}^{(m-1)}(\pi)$
or 
$\lambda \hookrightarrow 
L(\Delta_\rho[x-1,-(x-1)], \rho|\cdot|^x) \times (\rho|\cdot|^x)^{m-1} \rtimes D_{\rho|\cdot|^x}^{(m-1)}(\pi)$.
Since $x \geq1$ and $D_{\rho|\cdot|^x}^{(m)}(\lambda) \not=0$, the former case must hold.
Moreover, we see that 
\[
\lambda \hookrightarrow \Delta_\rho[x,-(x-1)] \rtimes \pi.
\]

\item
Since $\Delta_\rho[x,-(x-1)] \times \pi(\phi, \eta) \twoheadrightarrow \pi$, 
we see that $\lambda$ is the unique irreducible (Langlands) quotient of $\Delta_\rho[x,-(x-1)] \rtimes \pi$. 
Since this quotient appear in subquotients with multiplicity one, 
$\Delta_\rho[x,-(x-1)] \rtimes \pi$ must be irreducible. 
\end{itemize}
This completes the proof.
\end{proof}

%\subsection{M{\oe}glin's construction}
\subsection{M{\oe}glin's construction}\label{s.moe}
To obtain Theorem \ref{mult1}, M{\oe}glin constructed $A$-packets $\Pi_\psi$ concretely. 
In this subsection, we review her construction in a special case. 
For more precision, see \cite{Moe1, Moe2, Moe3} and \cite{X2}.
\par

Let $\psi \in \Psi_\gp(G_n)$. 
We assume that $\psi$ is of the form 
\[
\psi = \phi_0 + \left( \bigoplus_{i=1}^r \rho \boxtimes S_{a_i} \boxtimes S_{b_i} \right)
\]
such that
\begin{enumerate}
\item[(a)]
$\phi_0 \in \Phi_\gp(G)$ such that $\phi_0 \not\supset \rho \boxtimes S_d$ for any $d \geq 1$; 
\item[(b)]
$a_i \geq b_i$ for any $1 \leq i \leq r$; 
\item[(c)]
if $a_i-b_i > a_j - b_j$ and $a_i + b_i > a_j + b_j$, then $i > j$. 
\end{enumerate}
Note that the last condition may not determine an order on $\{(a_1,b_1), \dots, (a_r,b_r)\}$ uniquely. 
Once we fix such an order, 
we write $\rho \boxtimes S_{a_i} \boxtimes S_{b_i} <_\psi \rho \boxtimes S_{a_j} \boxtimes S_{b_j}$ if $i<j$.
\par

Fix $\eta_0 \in \widehat{\AA_{\phi_0}}$ such that 
if $\phi_0 = \oplus_i \phi_i$ is a decomposition into irreducible representations, 
then $\eta_0(\psi_i) = \eta_0(\psi_j)$ whenever $\phi_i \cong \phi_j$.
Let $\ub{l} = (l_1, \dots, l_r) \in \Z^r$ and $\ub{\eta} = (\eta_1, \dots, \eta_r) \in \{\pm1\}^r$ such that 
$0 \leq l_i \leq b_i/2$ and  
\[
\eta_0(z_{\phi_0}) 
\prod_{i=1}^r (-1)^{[b_i/2]+l_i}\eta_i^{b_i} = 1. 
\]
For these data, 
M{\oe}glin constructed a representation $\pi_{<_\psi}(\psi, \ub{l}, \ub{\eta}, \eta_0)$ of $G_n$. 
If the order $<_\psi$ is fixed, 
we may write $\pi(\psi, \ub{l}, \ub{\eta}, \eta_0) = \pi_{<_\psi}(\psi, \ub{l}, \ub{\eta}, \eta_0)$.

\begin{thm}[M{\oe}glin]\label{moe}
Notation is as above. 

\begin{enumerate}
\item
The representation $\pi(\psi, \ub{l}, \ub{\eta}, \eta_0)$ is irreducible or zero.
Moreover, 
\[
\Pi_\psi = \{ \pi(\psi, \ub{l}, \ub{\eta}, \eta_0) \;|\; \text{$\ub{l}, \ub{\eta}, \eta_0$ as above}\} \setminus \{0\}. 
\]

\item
If $\pi(\psi, \ub{l}, \ub{\eta}, \eta_0) \cong \pi(\psi, \ub{l'}, \ub{\eta'}, \eta'_0) \not= 0$, 
then $\ub{l} = \ub{l'}$, $\eta_0 = \eta_0'$, and $\eta_i = \eta_i'$ unless $l_i = b_i/2$. 

\item
Assume further that $a_i-b_i \geq a_{i-1}+b_{i-1}$ for any $i > 1$. 
Then $\pi(\psi, \ub{l}, \ub{\eta}, \eta_0)$ is nonzero and is a unique irreducible subrepresentation of 
\[
\bigtimes_{i=1}^r 
L\left(\Delta_{\rho}\left[ \half{a_i-b_i}, -\half{a_i+b_i}+1\right], 
\dots, 
\Delta_{\rho}\left[ \half{a_i-b_i}+l_i-1, -\half{a_i+b_i}+l_i\right] \right)
\rtimes \pi(\phi, \eta), 
\]
where 
\[
\phi = \phi_0 + \left(
\bigoplus_{i=1}^r \rho \boxtimes (S_{a_i-b_i+2l_i+1} + \dots + S_{a_i+b_i-2l_i-1})
\right),
\]
and $\eta \in \widehat{\Sc_\phi} \subset \widehat{\AA_\phi}$ is given by 
$\eta|\AA_{\phi_0} = \eta_0$ and 
$\eta(\rho \boxtimes S_{a_i-b_i+2l_i+2c-1}) = (-1)^{c-1} \eta_i$ for $1 \leq c \leq b_i - 2l_i$.

\item
Suppose that 
\[
\psi' = \phi_0 + \left( \bigoplus_{i=1}^r \rho \boxtimes S_{a'_i} \boxtimes S_{b_i} \right) \in \Psi_\gp(G)
\]
also satisfies the above assumptions (a)--(c).
Assume further that $a'_i \geq a_i$ for any $i \geq 1$ 
and that $a'_i-b_i \geq a'_{i-1}+b_{i-1}$ for any $i > 1$. 
Then 
\[
\pi(\psi, \ub{l}, \ub{\eta}, \eta_0) = 
\Jac_{\rho|\cdot|^{\half{a'_r-1}}, \dots, \rho|\cdot|^{\half{a_r+1}}}
\circ \dots \circ
\Jac_{\rho|\cdot|^{\half{a'_1-1}}, \dots, \rho|\cdot|^{\half{a_1+1}}}
(\pi(\psi', \ub{l}, \ub{\eta}, \eta_0)).
\]
\end{enumerate}
\end{thm}
For the proof, see \cite[\S 8]{X2}.

\begin{ex}\label{ex.xu1'}
Let us consider the situation of Example \ref{ex.xu1} with $x \geq 1$. 
Set $\psi = \phi + (\rho \boxtimes S_{2x} \boxtimes S_2)^b$. 
Write 
\[
\psi = \phi_0 + \left( \bigoplus_{i=1}^r \rho \boxtimes S_{a_i} \boxtimes S_{b_i} \right)
\]
as above such that $a_1 \leq \dots \leq a_r$. 
Define 
\begin{itemize}
\item
$\ub{l} \in \Z^r$ so that $l_i = 1$ if $b_i = 2$, and $l_i = 0$ if $b_i=1$; 
\item
$\ub{\eta} \in \{\pm1\}^r$ so that $\eta_i = \eta(\rho \boxtimes S_{a_i})$ if $b_i = 1$; 
\item
$\eta_0 \in \widehat{\AA_{\phi_0}}$ by $\eta_0 = \eta|\AA_{\phi_0}$.
\end{itemize}
Then it is easy to see that $\pi(\psi, \ub{l}, \ub{\eta}, \eta_0) = L(\Delta_\rho[x-1,-x]^b; \pi(\phi, \eta))$. 
\end{ex}

Similarly, we have the following proposition, which is a key observation.
\begin{prop}\label{ex.xu2}
Let $\phi \in \Phi_\gp(G)$ and $\eta \in \widehat{\Sc_\phi}$. 
Fix $x \in (1/2)\Z$ with $x \geq 1$. 
Suppose that $\rho \boxtimes S_{2x+1}$ is self-dual of the same type as $\phi$.
Assume that $m_{\phi}(\rho \boxtimes S_{2x+1}) \not=0$, $m_{\phi}(\rho \boxtimes S_{2x-1}) \not=0$
and 
\[
\eta(\rho \boxtimes S_{2x+1})\eta(\rho \boxtimes S_{2x-1}) = (-1)^{b+1}.
\]
Then $L(\Delta_\rho[x-1,-x]^b; \pi(\phi, \eta)) \in \Pi_\psi$ with 
\[
\psi = \phi - \rho \boxtimes (S_{2x+1}+S_{2x-1}) + (\rho \boxtimes S_{2x} \boxtimes S_2)^{b+1}.
\]
More precisely, 
when we write 
\[
\psi = \phi_0 + \left( \bigoplus_{i=1}^r \rho \boxtimes S_{a_i} \boxtimes S_{b_i} \right)
\]
as above such that $a_1 \leq \dots \leq a_r$, 
define 
\begin{itemize}
\item
$\ub{l} \in \Z^r$ so that $l_i = 0$ for any $1 \leq i \leq r$; 
\item
$\ub{\eta} \in \{\pm1\}^r$ so that 
\[
\eta_i = 
\left\{
\begin{aligned}
&\eta(\rho \boxtimes S_{a_i}) \iif b_i = 1, \\
&(-1)^{\#\{j < i \;|\; b_j = 2 \}} \eta(\rho \boxtimes S_{2x-1}) \iif b_i =2; 
\end{aligned}
\right. 
\]
\item
$\eta_0 \in \widehat{\AA_{\phi_0}}$ by $\eta_0 = \eta|\AA_{\phi_0}$.
\end{itemize}
Then $\pi(\psi, \ub{l}, \ub{\eta}, \eta_0) = L(\Delta_\rho[x-1,-x]^b; \pi(\phi, \eta))$.
In particular, if $D_{\rho|\cdot|^x}^{(k)}(L(\Delta_\rho[x-1,-x]^b; \pi(\phi, \eta))) \not= 0$, 
then $k \leq m_\phi(\rho \boxtimes S_{2x+1})-1$.
\end{prop}
\begin{proof}
When $b=0$ or $b=1$, it follows from Theorem \ref{moe} together with Proposition \ref{error}.
Suppose that $b \geq 2$. 
Take
\begin{align*}
\psi' = \psi 
&- \rho \boxtimes S_{2x} \boxtimes S_2 + \rho \boxtimes S_{2x+2} \boxtimes S_2 
-\left(\bigoplus_{\substack{1 \leq i \leq r \\ a_i \geq 2x+1}} \rho \boxtimes S_{a_i}\right) 
+\left(\bigoplus_{\substack{1 \leq i \leq r \\ a_i \geq 2x+1}} \rho \boxtimes S_{a'_i}\right)
\end{align*}
such that $a'_r > a'_{r-1} > \cdots$ are sufficiently large
so that $m_{\psi'}(\rho \boxtimes S_{2x+1}) = m_{\psi'}(\rho \boxtimes S_{2x+3}) = 0$.
By Theorem \ref{moe} (4), we have
$\pi(\psi, \ub{l}, \ub{\eta}, \eta_0) = J_2 \circ J_1(\pi(\psi', \ub{l}, \ub{\eta}, \eta_0))$ with 
\begin{align*}
J_1 &= \Jac_{\rho|\cdot|^{x}, \rho|\cdot|^{x+1}}, \\
J_2 &= \left( \Jac_{\rho|\cdot|^{\half{a'_{r}-1}}, \dots, \rho|\cdot|^{\half{a_{r}+1}}} \right)
\circ 
\left( \Jac_{\rho|\cdot|^{\half{a'_{r-1}-1}}, \dots, \rho|\cdot|^{\half{a_{r-1}+1}}} \right)
\circ \cdots.
\end{align*}
Since $l_{i} = 0$ for any $i$, we may apply the induction hypothesis to $\pi(\psi', \ub{l}, \ub{\eta}, \eta_0)$. 
Hence $\pi(\psi', \ub{l}, \ub{\eta}, \eta_0) = L(\Delta_{\rho}[x-1, -x]^{b-1}; \pi(\phi', \eta'))$, 
where 
\begin{align*}
\phi' 
&= 
\psi' - (\rho \boxtimes S_{2x} \boxtimes S_2)^b 
- \rho \boxtimes S_{2x+2} \boxtimes S_2 + \rho \boxtimes (S_{2x-1} + S_{2x+1}^2 + S_{2x+3})
\\&=
\phi + \rho \boxtimes (S_{2x+1} + S_{2x+3}) 
-\left(\bigoplus_{\substack{1 \leq i \leq r \\ a_i \geq 2x+1}} \rho \boxtimes S_{a_i}\right) 
+\left(\bigoplus_{\substack{1 \leq i \leq r \\ a_i \geq 2x+1}} \rho \boxtimes S_{a'_i}\right).
\end{align*}
Note that $m_{\phi'}(\rho \boxtimes S_{2x+1}) = 2$, $m_{\phi'}(\rho \boxtimes S_{2x+3}) = 1$
and $\eta'(\rho \boxtimes S_{2x-1}) \eta'(\rho \boxtimes S_{2x+1}) = (-1)^b$. 
\par

Take $\kappa \in \{0,1\}$ such that $b \equiv \kappa \bmod 2$. 
By Proposition \ref{irr01}, 
$\Delta_\rho[x-1,-x] \rtimes L(\Delta_{\rho}[x-1, -x]^{1-\kappa}; \pi(\phi', \eta'))$ is irreducible. 
Note that any irreducible subquotient of $\Delta_\rho[x-1,-x]^{2c}$ is of the form
\[
L(\Delta_\rho[x-1,-x]; \Delta_\rho[x,-(x-1)])^{\alpha} 
\times 
(\Delta_\rho[x,-x] \times \Delta_\rho[x-1,-(x-1)])^\beta
\]
with $\alpha+\beta = c$.
Considering the Langlands data for $\pi(\psi', \ub{l}, \ub{\eta}, \eta_0)$, 
we see that it is a subrepresentation of 
\[
L(\Delta_\rho[x-1,-x]; \Delta_\rho[x,-(x-1)])^{\half{b-2+\kappa}} 
\rtimes L(\Delta_\rho[x-1,-x]^{1-\kappa}; \pi(\phi', \eta')). 
\]
Note that this induced representation is irreducible. 
Indeed, it is unitary so that it is semisimple, 
but since it is subrepresentation of a standard module, it has a unique irreducible subrepresentation.
Hence $\pi(\psi, \ub{l}, \ub{\eta}, \eta_0)$ is equal to 
\[
L(\Delta_\rho[x-1,-x]; \Delta_\rho[x,-(x-1)])^{\half{b-2+\kappa}} 
\rtimes 
J_2 \circ J_1 (L(\Delta_\rho[x,-(x-1)]^{1-\kappa}; \pi(\phi', \eta'))). 
\]
Since $\Jac_{\rho|\cdot|^{x+1}}(\pi(\phi', \eta')) = 0$, 
we have
\[
J_2 \circ J_1 (L(\Delta_\rho[x,-(x-1)]^{1-\kappa}; \pi(\phi', \eta')))
= L(\Delta_\rho[x,-(x-1)]^{1-\kappa}; J_2 \circ J_1(\pi(\phi', \eta'))).
\]
By \cite[Theorem 4.2]{At}, we have 
\[
J_2 \circ J_1(\pi(\phi', \eta')) = L(\Delta_\rho[x,-(x-1)]; \pi(\phi, \eta)).
\]
Therefore, 
\begin{align*}
\pi(\psi, \ub{l}, \ub{\eta}, \eta_0) 
&= 
L(\Delta_\rho[x-1,-x]; \Delta_\rho[x,-(x-1)])^{\half{b-2+\kappa}} 
\rtimes 
L(\Delta_\rho[x,-(x-1)]^{2-\kappa}; \pi(\phi, \eta))
\\&\cong
L(\Delta_\rho[x,-(x-1)]^{b}; \pi(\phi, \eta)), 
\end{align*}
as desired. 
\end{proof}

When $x=1$, we have the following more general version. 
\begin{prop}\label{ex.xu3}
Let $\phi \in \Phi_\gp(G)$ and $\eta \in \widehat{\Sc_\phi}$. 
Suppose that $\rho$ is self-dual of the same type as $\phi$. 
Assume that $m_\phi(\rho) \not= 0 $, $m_\phi(\rho \boxtimes S_3) \not= 0$ 
and 
\[
\eta(\rho) \eta(\rho \boxtimes S_3) = (-1)^{b+1}.
\]
If $a \leq m_\phi(\rho)-1$, then $L((\rho|\cdot|^{-1})^a, \Delta_\rho[0,-1]^b; \pi(\phi, \eta)) \in \Pi_\psi$
with 
\begin{align*}
\psi = \phi &- \rho^a + (\rho \boxtimes S_1 \boxtimes S_3)^a 
\\&- (\rho + \rho \boxtimes S_3) + (\rho \boxtimes S_2 \boxtimes S_2)^{b+1}.
\end{align*}
In particular, if $D_{\rho|\cdot|^1}^{(k)}(\pi) \not= 0$, then $k \leq m_\phi(\rho \boxtimes S_{3})-1$.
\end{prop}
The proof is the same as the one of Proposition \ref{ex.xu2}, 
but it requires M{\oe}glin's construction more generally 
since $\psi$ contains $\rho \boxtimes S_1 \boxtimes S_3$. 
We omit the detail. 

%\section{Derivatives of certain representations}
%\section{Derivatives of certain representations}
\section{Derivatives of certain representations}\label{s.main}
Let $\phi \in \Phi_\gp(G)$ and $\eta \in \widehat{\Sc_\phi}$. 
Fix $x \in (1/2)\Z$ with $x > 0$ such that 
$\rho \boxtimes S_{2x+1}$ is self-dual of the same type as $\phi$. 
Let $\pi = L((\rho|\cdot|^{-x})^a, \Delta_\rho[x-1,-x]^b; \pi(\phi,\eta))$. 
By Jantzen's algorithm recalled in \S \ref{alg}, 
the computation of the highest derivatives of arbitrary irreducible representations
is reduced to the one of $D_{\rho|\cdot|^x}^{(k)}(\pi)$. 
In this section, we give an algorithm to compute this for $x \geq 1$.
\par

%\subsection{Statements}
\subsection{Statements}
The following is an algorithm. 
\begin{thm}\label{main}
Suppose that $x \geq 1$. 
\begin{enumerate}
\item
If $m_{\phi}(\rho \boxtimes S_{2x+1}) \not=0$, $m_{\phi}(\rho \boxtimes S_{2x-1}) \not=0$
and $\eta(\rho \boxtimes S_{2x+1})\eta(\rho \boxtimes S_{2x-1}) = (-1)^{b+1}$,  
set 
\[
\psi = \phi - \rho \boxtimes (S_{2x+1}+S_{2x-1}) + (\rho \boxtimes S_{2x} \boxtimes S_2)^{b+1}.
\]
Otherwise, set $\psi = \phi + (\rho \boxtimes S_{2x} \boxtimes S_2)^b$. 
Then $L(\Delta_\rho[x-1,-x]^b; \pi(\phi,\eta)) \in \Pi_\psi$. 

\item
Put $m = m_{\psi}(\rho \boxtimes S_{2x+1})$, $m' = m_{\psi}(\rho \boxtimes S_{2x-1})$, 
and 
\[
l = \min\{a-m',0\}. 
\]
Then $D_{\rho|\cdot|^x}^{(l+m)}(L((\rho|\cdot|^{-x})^a, \Delta_\rho[x-1,-x]^b; \pi(\phi,\eta)))$
is the highest derivative. 
Moreover, it is a unique irreducible subrepresentation of 
\[
(\rho|\cdot|^{-x})^{a-l} \rtimes D_{\rho|\cdot|^x}^{(m)}(L(\Delta_\rho[x-1,-x]^b; \pi(\phi,\eta))). 
\]

\item
If we set $\psi' = \psi - (\rho \boxtimes S_{2x+1})^m + (\rho \boxtimes S_{2x-1})^m$, 
then $D_{\rho|\cdot|^x}^{(m)}(L(\Delta_\rho[x-1,-x]^b; \pi(\phi,\eta))) \in \Pi_{\psi'}$. 
In particular, if we write
\[
\psi = \phi_0 + \left( \bigoplus_{i=1}^r \rho \boxtimes S_{a_i} \boxtimes S_{b_i} \right), 
\quad
\psi' = \phi_0 + \left( \bigoplus_{i=1}^r \rho \boxtimes S_{a'_i} \boxtimes S_{b'_i} \right)
\]
such that 
$a_1 \leq \dots \leq a_r$, $a'_1 \leq \dots \leq a'_r$, and 
such that $\phi_0 \not\supset \rho \boxtimes S_d$ for any $d > 0$, 
then 
\begin{align*}
\pi(\psi, \ub{l}, \ub{\eta}, \eta_0) &= L(\Delta_\rho[x-1,-x]^b; \pi(\phi,\eta)), \\
\pi(\psi', \ub{l'}, \ub{\eta'}, \eta'_0) &= D_{\rho|\cdot|^x}^{(m)}(L(\Delta_\rho[x-1,-x]^b; \pi(\phi,\eta)))
\end{align*}
for some data. 
These are related as follows:
\begin{itemize}
\item
$\eta_0 = \eta'_0 = \eta|\AA_{\phi_0}$; 
\item
if $b_i'=1$, then $l_i' = 0$; 
\item
if $b'_i = b_j = 2$, then $l'_i = l_j$ if and only if $m$ is even; 
\item
if $b'_i = 1$ and $a'_i \not= 2x-1$, then $\eta'_i = \eta(\rho \boxtimes S_{a_i})$; 
\item
if $b'_i = 1$ and $a'_i = 2x-1$, or if $b'_i = 2$ and $l'_i = 0$, 
then 
\[
\eta'_i = \eta(\rho \boxtimes S_{2x-1})
\quad\text{or} \quad
\eta'_i = (-1)^{m_\psi(\rho \boxtimes S_{2x} \boxtimes S_2)} \eta(\rho \boxtimes S_{2x+1}).
\]
These values are equal to each other if $\phi \supset \rho \boxtimes (S_{2x-1}+S_{2x+1})$.
\end{itemize}

\item
By Example \ref{ex.xu1'} and Proposition \ref{ex.xu2}, 
we can obtain the Langlands data for $D_{\rho|\cdot|^x}^{(m)}(L(\Delta_\rho[x-1,-x]^b; \pi(\phi,\eta)))$.
\end{enumerate}
\end{thm}

We can also write down this theorem more explicitly. 
Let $\kappa \in \{0,1\}$ such that $\kappa \equiv b \bmod 2$. 
Set $m_{d} = m_\phi(\rho \boxtimes S_{d})$ for $d \geq 1$.
For $d \geq 3$, define 
\[
\delta_d = \left\{
\begin{aligned}
&1	\iif \text{$m_dm_{d-2} \not= 0$ and $\eta(\rho \boxtimes S_{d}) \not= \eta(\rho \boxtimes S_{d-2})$}, \\
&0	\other.
\end{aligned}
\right. 
\]
Put
\[
\phi' = \phi - (\rho \boxtimes S_{2x+1})^{m_{2x+1}} + (\rho \boxtimes S_{2x-1})^{m_{2x+1}}.
\] 
Let $\eta' \in \widehat{\AA_{\phi'}}$ be the pullback of $\eta$ via the canonical map $\AA_{\phi'} \hookrightarrow \AA_\phi \twoheadrightarrow \Sc_{\phi}$. 
The following is a reformulation of Theorem \ref{main}, 
which follows from Example \ref{ex.xu1'} and Proposition \ref{ex.xu2} immediately. 

\begin{cor}\label{ab}
The notation is as above. 
Let $\pi = L((\rho|\cdot|^{-x})^a, \Delta_\rho[x-1,-x]^b; \pi(\phi,\eta))$. 
Then $D_{\rho|\cdot|^x}^{(k)}(\pi)$ is the highest derivative with 
\[
k = 
\left\{
\begin{aligned}
&\max\{a-m_{2x-1}+m_{2x+1}, m_{2x+1}-1\} 
\iif m_{2x-1}m_{2x+1} \not=0, \kappa \not= \delta_{2x+1}, \\
&\max\{a-m_{2x-1}+m_{2x+1}, m_{2x+1}\} 
\other.
\end{aligned}
\right. 
\]
Moreover: 
\begin{enumerate}
\item
Suppose that $m_{2x+1} = 0$. 
Set $l = \max\{a-m_{2x-1}, 0\}$. 
Then 
\[
D_{\rho|\cdot|^{x}}^{(l)}(\pi) = L((\rho|\cdot|^{-x})^{a-l}, \Delta_\rho[x-1,-x]^b; \pi(\phi,\eta)).
\]

\item
Suppose that $m_{2x-1}m_{2x+1} \not= 0$. 
Set 
\[
l = \left\{
\begin{aligned}
&\max\{a-m_{2x-1},0\} \iif \kappa = \delta_{2x+1}, \\
&\max\{a-m_{2x-1}+1,0\} \iif \kappa \not= \delta_{2x+1}.
\end{aligned}
\right. 
\]
\begin{enumerate}
\item
If $\kappa \not= \delta_{2x+1}$, and if $m_{2x+1}$ is odd, then
\[
D_{\rho|\cdot|^x}^{(l+m_{2x+1}-1)}(\pi) 
= 
L\left((\rho|\cdot|^{-x})^{a-l}, \Delta_\rho[x-1,-x]^b; \pi(\phi' + \rho \boxtimes (S_{2x+1}-S_{2x-1}), \eta')\right).
\]

\item
If $\kappa \not= \delta_{2x+1}$, and if $m_{2x+1}$ is even, then
\[
D_{\rho|\cdot|^x}^{(l+m_{2x+1}-1)}(\pi) 
= 
L\left((\rho|\cdot|^{-x})^{a-l}, \Delta_\rho[x-1,-x]^{b+1}; \pi(\phi'-(\rho \boxtimes S_{2x-1})^2, \eta')\right).
\]

\item
If $\kappa = \delta_{2x+1}$, and if $m_{2x+1}$ is odd and $b>0$, then
\[
D_{\rho|\cdot|^x}^{(l+m_{2x+1})}(\pi) 
= 
L\left((\rho|\cdot|^{-x})^{a-l}, \Delta_\rho[x-1,-x]^{b-1}; 
\pi(\phi' + \rho\boxtimes (S_{2x+1} + S_{2x-1}), \eta') \right).
\]

\item
If $\kappa = \delta_{2x+1}$, and if $m_{2x+1}$ is even or $b=0$, then
\[
D_{\rho|\cdot|^x}^{(l+m_{2x+1})}(\pi) 
=
L\left((\rho|\cdot|^{-x})^{a-l}, \Delta_\rho[x-1,-x]^{b}; \pi(\phi', \eta')\right).
\]
\end{enumerate}

\item
Suppose that $m_{2x-1} = 0$. 
\begin{enumerate}
\item
If $m_{2x+1}$ is even or $b=0$, then
\[
D_{\rho|\cdot|^x}^{(a+m_{2x+1})}(\pi) 
= 
L(\Delta_\rho[x-1,-x]^b; \pi(\phi', \eta'_b)),
\]
where $\eta_b'(\rho \boxtimes S_{2x-1}) = (-1)^b \eta(\rho \boxtimes S_{2x+1})$.

\item
If $m_{2x+1}$ is odd and $b>0$, then 
\[
D_{\rho|\cdot|^x}^{(a+m_{2x+1})}(\pi) 
= 
L(\Delta_\rho[x-1,-x]^{b-1}; \pi(\phi'+\rho \boxtimes (S_{2x-1}+S_{2x+1}), \eta'_b)),
\]
where $\eta_b'(\rho \boxtimes S_{2x-1}) = (-1)^b \eta(\rho \boxtimes S_{2x+1})$ and 
$\eta_b'(\rho \boxtimes S_{2x+1}) = \eta(\rho \boxtimes S_{2x+1})$.
\end{enumerate}
\end{enumerate}
\end{cor}

\begin{rem}
When $x=1/2$, 
we formally understand that $m_0 = 1$, $\eta(\rho \boxtimes S_0) = +1$ and $a=0$. 
Then Corollary \ref{ab} still holds for $x=1/2$ (\cite[Theorem 3.3]{J3}).
Note that when $x=1/2$, the case (3) does not appear.
\end{rem}

\par

Recall that for $k' \geq 0$, 
the parabolically induced representation $(\rho|\cdot|^{x})^{k'} \rtimes D_{\rho|\cdot|^x}^{(k)}(\pi)$
has a unique irreducible subrepresentation $\pi'$. 
When $k' = 0$ (\resp $k'=k$), we have $\pi' = D_{\rho|\cdot|^x}^{(k)}(\pi)$ (\resp $\pi'=\pi$). 
When $0 < k' < k$, the following formula follows from Corollary \ref{ab} immediately. 

\begin{cor}\label{k'}
Let $\pi = L((\rho|\cdot|^{-x})^a, \Delta_\rho[x-1,-x]^b; \pi(\phi, \eta))$ be as in Corollary \ref{ab}, 
and $D_{\rho|\cdot|^x}^{(k)}(\pi)$ be its highest derivative. 
Set 
\[
a_0 = \left\{
\begin{aligned}
&\min\{a,m_{2x-1}\} \iif \kappa = \delta_{2x+1}, \\
&\min\{a,m_{2x-1}-1\} \iif \kappa \not= \delta_{2x+1}. 
\end{aligned}
\right.
\]
Suppose that $0 < k' < k$. 
Then the unique irreducible subrepresentation $\pi'$ 
of $(\rho|\cdot|^{x})^{k'} \rtimes D_{\rho|\cdot|^x}^{(k)}(\pi)$
is of the form $L((\rho|\cdot|^{-x})^{a_{k'}}, \Delta_\rho[x-1,-x]^{b_{k'}}; \pi(\phi_{k'}, \eta_{k'}))$, 
where $(a_{k'}, b_{k'}, \phi_{k'}, \eta_{k'})$ is given as follows. 
 
\begin{enumerate}
\item
Suppose that $m_{2x+1} = 0$. 
Then $k = a - m_{2x-1} > 0$ and 
\[
(a_{k'}, b_{k'}, \phi_{k'}, \eta_{k'}) = (k'+m_{2x-1}, b, \phi, \eta). 
\]

\item
Suppose that $m_{2x-1}m_{2x+1} \not= 0$. 

\begin{enumerate}
\item
If $\kappa \not= \delta_{2x+1}$, 
then $(a_{k'}, b_{k'}, \phi_{k'}, \eta_{k'})$ is equal to
\[
\left\{
\begin{aligned}
&(a_0, b, \phi' + (\rho \boxtimes S_{2x+1})^{k'+1}-(\rho \boxtimes S_{2x-1})^{k'+1},\eta') 
\iif k' < m_{2x+1}-1, k' \not\equiv m_{2x+1} \bmod 2, \\
&(a_0, b+1, \phi' + (\rho \boxtimes S_{2x+1})^{k'}-(\rho \boxtimes S_{2x-1})^{k'+2},\eta') 
\iif k' < m_{2x+1}-1, k' \equiv m_{2x+1} \bmod 2, \\
&(k'-m_{2x+1}+m_{2x-1}, b, \phi, \eta) 
\iif k' \geq m_{2x+1}-1.
\end{aligned}
\right. 
\]

\item
If $\kappa = \delta_{2x+1}$ and $b>0$, 
then $(a_{k'}, b_{k'}, \phi_{k'}, \eta_{k'})$ is equal to
\[
\left\{
\begin{aligned}
&(a_0, b, \phi' + (\rho \boxtimes S_{2x+1})^{k'}-(\rho \boxtimes S_{2x-1})^{k'},\eta') 
\iif k' < m_{2x+1}, k' \equiv m_{2x+1} \bmod 2, \\
&(a_0, b-1, \phi' + (\rho \boxtimes S_{2x+1})^{k'+1}-(\rho \boxtimes S_{2x-1})^{k'-1},\eta') 
\iif k' < m_{2x+1}, k' \not\equiv m_{2x+1} \bmod 2, \\
&(k'-m_{2x+1}+m_{2x-1}, b, \phi, \eta) 
\iif k' \geq m_{2x+1}.
\end{aligned}
\right. 
\]
Similarly, if $\kappa = \delta_{2x+1}$ and $b=0$, 
then $(a_{k'}, b_{k'}, \phi_{k'}, \eta_{k'})$ is equal to
\[
\left\{
\begin{aligned}
&(a_0, 0, \phi' + (\rho \boxtimes S_{2x+1})^{k'}-(\rho \boxtimes S_{2x-1})^{k'},\eta') 
\iif k' < m_{2x+1}, \\
&(k'-m_{2x+1}+m_{2x-1}, 0, \phi, \eta) 
\iif k' \geq m_{2x+1}.
\end{aligned}
\right. 
\]
\end{enumerate}

\item
Suppose that $m_{2x-1} = 0$. 
If $b>0$, 
then $(a_{k'}, b_{k'}, \phi_{k'}, \eta_{k'})$ is equal to
\[
\left\{
\begin{aligned}
&(0, b, \phi' + (\rho \boxtimes S_{2x+1})^{k'}-(\rho \boxtimes S_{2x-1})^{k'},\eta_b') 
\iif k' < m_{2x+1}, k' \equiv m_{2x+1} \bmod 2, \\
&(0, b-1, \phi' + (\rho \boxtimes S_{2x+1})^{k'+1}-(\rho \boxtimes S_{2x-1})^{k'-1},\eta_b') 
\iif k' < m_{2x+1}, k' \not\equiv m_{2x+1} \bmod 2, \\
&(k'-m_{2x+1}, b, \phi, \eta) 
\iif k' \geq m_{2x+1}.
\end{aligned}
\right. 
\]
Similarly, if $b=0$, 
then $(a_{k'}, b_{k'}, \phi_{k'}, \eta_{k'})$ is equal to
\[
\left\{
\begin{aligned}
&(0, 0, \phi' + (\rho \boxtimes S_{2x+1})^{k'}-(\rho \boxtimes S_{2x-1})^{k'},\eta_b') 
\iif k' < m_{2x+1}, \\
&(k'-m_{2x+1}, 0, \phi, \eta) 
\iif k' \geq m_{2x+1}.
\end{aligned}
\right. 
\]
\end{enumerate}
\end{cor}

%\subsection{Examples of the Aubert dual}\label{s.ex}
\subsection{Examples of the Aubert duals}\label{s.ex}
Let $\pi$ be an irreducible representation of $G_n$. 
By Theorem \ref{vs} and Proposition \ref{ML}, 
if one could compute $M^+_{\rho^\vee}(\pi)$ for all $\rho$, 
one would obtain the Langlands data of the Aubert dual $\hat\pi$ of $\pi$ almost explicitly. 
In this subsection, we give some specific examples of explicit calculations of $M^+_{\rho^\vee}(\pi)$ 
and $\hat\pi$ for some irreducible representations $\pi$ of $G_n = \Sp_{2n}(F)$.  
When $\phi = S_{d_1} \oplus \dots \oplus S_{d_t} \in \Phi(G)$ and $\eta_i = \eta(S_{d_i})$, 
where $d_1, \dots, d_t$ are all odd, 
we write
\[
\pi(\phi,\eta) = \pi(d_1^{\eta_1}, \dots, d_t^{\eta_t}). 
\]
We also write $\1 = \1_{\GL_1(F)}$, and $\Delta[x,y] = \Delta_\1[x,y]$.
\par

First, let us consider $\pi(1^+,3^-,5^+,5^+,5^+,5^+,7^-) \in \Irr(\Sp_{30}(F))$. 
Then
\begin{align*}
&
M^+_\1\left(\pi(1^+,3^-,5^+,5^+,5^+,5^+,7^-)\right)\notag
\\&=
\left[
(2,3); 
M^+_\1\left(
L(\Delta[1,-2]; \pi(1^+,3^-,3^-,3^-,7^-))
\right)
\right]
\\&=
\left[
(2,3), (3,1); 
M^+_\1\left(
L(\Delta[1,-2]; \pi(1^+,3^-,3^-,3^-,5^-))
\right)
\right]
\\&=
\left[
(2,3), (3,1), (1,3); 
M^+_\1\left(
L(\Delta[0,-2]; \pi(1^+,1^+,1^+,3^-,5^-))
\right)
\right]
\\&\overset{(1)}{=}
\left[
(2,3), (3,1), (1,3), (2,1); 
M^+_\1\left(
L(\Delta[0,-2]; \pi(1^+,1^+,1^+,3^-,3^-))
\right)
\right]
\\&=
\left[
(2,3), (3,1), (1,3), (2,1), (0,2); 
M^+_\1\left(
L(\Delta[-1,-2]; \pi(1^+,3^-,3^-))
\right)
\right]
\\&=
\left[
(2,3), (3,1), (1,3), (2,1), (0,2), (1,1); 
M^+_\1\left(
L(\Delta[-1,-2], \Delta[0,-1]; \pi(1^+))
\right)
\right]
\\&\overset{(2)}{=}
\left[
(2,3), (3,1), (1,3), (2,1), (0,2), (1,1), (-1,1); 
M^+_\1\left(
L(|\cdot|^{-2}, \Delta[0,-1]; \pi(1^+))
\right)
\right]
\\&=
\left[
(2,3), (3,1), (1,3), (2,1), (0,2), (1,1), (-1,1), (0,1); 
M^+_\1\left(
L(|\cdot|^{-2}, |\cdot|^{-1}; \pi(1^+))
\right)
\right]
\\&=
\left[
(2,3), (3,1), (1,3), (2,1), (0,2), (1,1), (-1,1), (0,1), (-2,1), (-1,1); \pi(1^+)
\right].
\end{align*}
The equations (1) and (2) are non-trivial. 
We explain these equations.
\begin{enumerate}
\item[(1)]
We compute
the highest derivative $D_{|\cdot|^2}^{(k)}(L(\Delta[0,-2]; \pi(1^+,1^+,1^+,3^-,5^-)))$
by applying Jantzen's algorithm recalled in \S \ref{alg}.
By Jantzen's Claim 1, we have
\begin{align*}
L(\Delta[0,-2]; \pi(1^+,1^+,1^+,3^-,5^-))
\hookrightarrow \Delta[0,-1] \rtimes L(|\cdot|^{-2}; \pi(1^+,1^+,1^+,3^-,5^-)).
\end{align*}
By Corollary \ref{ab} (2)-(d), 
\[
D_{|\cdot|^2}^{(2)}(L(|\cdot|^{-2}; \pi(1^+,1^+,1^+,3^-,5^-)))
= L(|\cdot|^{-2}; \pi(1^+,1^+,1^+,3^-,3^-)). 
\]
Therefore, Jantzen's Claim 3 says that
\[
D_{|\cdot|^2}^{(2)}(L(\Delta[0,-2]; \pi(1^+,1^+,1^+,3^-,5^-))) = L(\Delta[0,-2]; \pi(1^+,1^+,1^+,3^-,3^-))
\]
is the highest derivative. 

\item[(2)]
We have to show that $\Jac_{|\cdot|^2}(L(\Delta[-1,-2], \Delta[0,-1]; \pi(1^+)))=0$. 
It follows from $R_{|\cdot|^{-2}}(L(\Delta[-1,-2], \Delta[0,-1]))=0$. 
\end{enumerate}
\par

Therefore, 
\begin{align*}
&M^-_\1\left(\hat\pi(1^+,3^-,5^+,5^+,5^+,5^+,7^-)\right)
\\&=
\left[
(-2,3), (-3,1), (-1,3), (-2,1), (0,2), (-1,1), (1,1), (0,1), (2,1), (1,1); \pi(1^+)
\right].
\end{align*}
By Theorem \ref{vs} and Proposition \ref{ML}, we have
\begin{align*}
&\hat\pi(1^+,3^-,5^+,5^+,5^+,5^+,7^-) \hookrightarrow
\\&
\Delta[-2,-3] \times (|\cdot|^{-2})^{2}
\times \Delta[-1,-2] \times (|\cdot|^{-1})^{2} \times \Delta[0,-1]
\rtimes \pi(1^\epsilon,1^\epsilon,1^\epsilon,3^\epsilon,5^+)
\end{align*}
for some sign $\epsilon \in \{\pm\}$.
Since
\[
\pi(1^\epsilon,3^\epsilon,5^+) = D_{|\cdot|^{-1}}^{(1)} \circ D_{|\cdot|^{0}}^{(2)} \circ D_{|\cdot|^{-2}}^{(1)} \circ D_{|\cdot|^{-1}}^{(3)} \circ D_{|\cdot|^{-3}}^{(1)} \circ D_{|\cdot|^{-2}}^{(3)}
(\hat\pi(1^+,3^-,5^+,5^+,5^+,5^+,7^-))
\]
up to multiplicity, 
we must have
\begin{align*}
\hat\pi(1^\epsilon,3^\epsilon,5^+) 
&= D_{|\cdot|^{1}}^{(1)} \circ D_{|\cdot|^{0}}^{(2)} \circ D_{|\cdot|^{2}}^{(1)} 
\circ D_{|\cdot|^{1}}^{(3)} \circ D_{|\cdot|^{3}}^{(1)} \circ D_{|\cdot|^{2}}^{(3)}
(\pi(1^+,3^-,5^+,5^+,5^+,5^+,7^-))
\\&=
L(\Delta[-1,-2], \Delta[0,-1]; \pi(1^+)). 
\end{align*}
\par

To determine $\epsilon \in \{\pm\}$, we compute the Aubert duals of 
$\pi(1^+,3^+,5^+)$ and $\pi(1^-,3^-,5^+)$. 
When $\epsilon = +$, we have
\begin{align*}
M^+_\1\left( \pi(1^+,3^+,5^+) \right)
&=
\left[
(2,1); M^+_\1\left( \pi(1^+,3^+,3^+) \right)
\right]
\\&=
\left[
(2,1), (1,2); M^+_\1\left( \pi(1^+,1^+,1^+) \right)
\right]
\\&=
\left[
(2,1), (1,2), (0,1); \pi(1^+)
\right].
\end{align*}
Hence
\begin{align*}
\hat\pi(1^+,3^+,5^+)
\hookrightarrow 
|\cdot|^{-2} \times (|\cdot|^{-1})^{2} \rtimes \pi(1^+,1^+,1^+).
\end{align*}
When $\epsilon = -$, we have
\begin{align*}
M^+_\1\left( \pi(1^-,3^-,5^+) \right)
&=
\left[
(1,1); M^+_\1\left( \pi(1^-,1^-,5^+) \right)
\right]
\\&=
\left[
(1,1), (2,1); M^+_\1\left( \pi(1^-,1^-,3^+) \right)
\right]
\\&=
\left[
(1,1), (2,1), (0,1); M^+_\1\left( \pi(3^+) \right)
\right]
\\&=
\left[
(1,1), (2,1), (0,1), (1,1); \pi(1^+)
\right].
\end{align*}
Hence
\begin{align*}
\hat\pi(1^-,3^-,5^+)
\hookrightarrow 
\Delta[-1,-2] \times \Delta[0,-1] \rtimes \pi(1^+). 
\end{align*}
Therefore, the correct sign is $\epsilon = -$. 
\par

This method does not always determine $\hat\pi$. 
For example, let us consider 
$\pi_\epsilon = L(\Delta[0,-2]; \pi(1^\epsilon,1^\epsilon,3^+)) \in \Irr(\Sp_{10}(F))$
for a sign $\epsilon \in \{\pm\}$.
Then for any $\epsilon \in \{\pm\}$, we have
\[
M^+_\1(\pi_\epsilon) = \left[(0,2), (1,1), (2,1), (-1,1); \pi(1^+)\right]. 
\]
Using Theorem \ref{vs} and Proposition \ref{ML}, 
this implies that $\hat\pi_\epsilon = \pi_{\hat\epsilon}$ for some $\hat\epsilon \in \{\pm\}$. 
However, the correspondence $\epsilon \mapsto \hat\epsilon$ is not determined.

%\section{Proof of Theorem \ref{main}}
%\section{Proof of Theorem \ref{main}}
\section{Proof of Theorem \ref{main}}\label{s.proof}
In this section, we prove Theorem \ref{main}.

%\subsection{The case $a=0$}
\subsection{The case $a=0$}\label{s.a=0}
First, we consider the case $a = 0$. 

\begin{proof}[Proof of Theorem \ref{main} when $a=0$]
Let $\pi = L(\Delta_\rho[x-1,-x]^b; \pi(\phi, \eta))$. 
The assertions (1) and (4) follow from Example \ref{ex.xu1'} and Proposition \ref{ex.xu2}. 
For (2) and (3), we note that $D_{\rho|\cdot|^x}^{(k)}(\pi) \not= 0 \implies k \leq m$ by Lemma \ref{xu}. 
\par

To determine $D_{\rho|\cdot|^x}^{(m)}(\pi)$, we will use M{\oe}glin's construction.
Write 
\begin{align*}
\psi = \phi_0 
&+ \rho \boxtimes (S_{a_{-t'}} + \dots + S_{a_{-1}}) 
\\&+(\rho \boxtimes S_{2x-1})^{m'}
+ (\rho \boxtimes S_{2x} \boxtimes S_2)^{b'}
+ (\rho \boxtimes S_{2x+1})^{m} 
\\&+ \rho \boxtimes (S_{a_{1}} + \dots + S_{a_{t}}), 
\end{align*}
where $a_{-t'} \leq \dots \leq a_{-1} < 2x-1$ and $2x+1 < a_1 \leq \dots \leq a_t$, 
and $\phi_0 \not\supset \rho \boxtimes S_{d}$ for any $d > 0$. 
Take an $A$-parameter of the form 
\begin{align*}
\psi_> = \phi_0 
&+ \rho \boxtimes (S_{a_{-t'}} + \dots + S_{a_{-1}}) 
\\&+(\rho \boxtimes S_{2x-1})^{m'}
+ (\rho \boxtimes S_{2x} \boxtimes S_2)^{b'}
+ \rho \boxtimes (S_{2y_1+1} + \dots + S_{2y_m+1})
\\&+ \rho \boxtimes (S_{a'_{1}} + \dots + S_{a'_{t}}), 
\end{align*}
where
$2y_i+1 \equiv a_i' \equiv 2x+1 \bmod 2$ such that 
$2x+1 < 2y_1+1 < \dots < 2y_m+1 < a_1' < \dots < a_t'$. 
When $\pi = \pi(\psi, \ub{l}, \ub{\eta}, \eta_0)$, 
we set
$\pi_> = \pi(\psi_>, \ub{l}, \ub{\eta}, \eta_0)$. 
Then M{\oe}glin's construction says that $\pi = J_2 \circ J_1 (\pi_>)$, 
where
\begin{align*}
J_1 &= 
\Jac_{\rho|\cdot|^{y_m}, \dots, \rho|\cdot|^{x+1}} 
\circ \dots \circ 
\Jac_{\rho|\cdot|^{y_1}, \dots, \rho|\cdot|^{x+1}}, \\
J_2 &= 
\Jac_{\rho|\cdot|^{\half{a_t'-1}}, \dots, \rho|\cdot|^{\half{a_t+1}}} 
\circ \dots \circ 
\Jac_{\rho|\cdot|^{\half{a_1'-1}}, \dots, \rho|\cdot|^{\half{a_1+1}}}.
\end{align*}
\par

Set $\pi' = J_2 \circ J_1'(\pi_>)$ with
\[
J_1' = \Jac_{\rho|\cdot|^{y_m}, \dots, \rho|\cdot|^{x}} 
\circ \dots \circ 
\Jac_{\rho|\cdot|^{y_1}, \dots, \rho|\cdot|^{x}}.
\]
Then $\pi' = \pi_{<_{\psi}}(\psi', \ub{l}, \ub{\eta}, \eta_0)$, 
where 
\begin{align*}
\psi' = \phi_0 
&+ \rho \boxtimes (S_{d_{-t'}} + \dots + S_{d_{-1}}) 
\\&+(\rho \boxtimes S_{2x-1})^{m'}
+ (\rho \boxtimes S_{2x} \boxtimes S_2)^{b'}
+ (\rho \boxtimes S_{2x-1})^{m} 
\\&+ \rho \boxtimes (S_{d_{1}} + \dots + S_{d_{t}}).
\end{align*}
However, the order $<_{\psi}$ for $(\rho \boxtimes S_{2x-1})^{m'+m}$ 
and $(\rho \boxtimes S_{2x} \boxtimes S_2)^{b'}$ is 
\[
\underbrace{\rho \boxtimes S_{2x-1} <_{\psi} \dots <_{\psi} \rho \boxtimes S_{2x-1}}_{m'}
<_{\psi} \rho \boxtimes S_{2x} \boxtimes S_2 <_{\psi}
\underbrace{\rho \boxtimes S_{2x-1} <_{\psi} \dots <_{\psi} \rho \boxtimes S_{2x-1}}_{m}.
\]
To change the order so that $\rho \boxtimes S_{2x-1} <_{\psi'} \rho \boxtimes S_{2x} \boxtimes S_2$
for all $\rho \boxtimes S_{2x-1}$ and $\rho \boxtimes S_{2x} \boxtimes S_2$, 
we use a result of Xu \cite[Theorem 6.1]{X3}. 
By this theorem, we see that $\pi' = \pi(\psi', \ub{l'}, \ub{\eta'}, \eta_0)$, 
where $\ub{l'}$ and $\ub{\eta'}$ are given in Theorem \ref{main} (3). 
In particular, $\pi' \not= 0$ by Example \ref{ex.xu1'} and Proposition \ref{ex.xu2}.
\par

Now, by \cite[Corollary 5.4, Lemma 5.7]{X1} and Lemma \ref{xu}, 
one can write
\[
\pi_> \hookrightarrow \bigtimes_{i=1}^m \Delta_\rho[y_i,x] \rtimes J'_1 (\pi_>). 
\]
Hence $J_1(\pi_>) \hookrightarrow (\rho|\cdot|^x)^m \rtimes J'_1 (\pi_>)$ so that
we conclude that
\[
\pi \hookrightarrow (\rho|\cdot|^x)^m \rtimes \pi'.
\]
Since $D_{\rho|\cdot|^{x}}^{(m)}(\pi)$ is irreducible or zero by Proposition \ref{highest}, 
we see that $D_{\rho|\cdot|^{x}}^{(m)}(\pi) = \pi'$.
\end{proof}

%\subsection{The case $x>1$}
\subsection{The case $x>1$}\label{s.x>1}
In this subsection, we will prove the following proposition.

\begin{prop}\label{x>1}
Assume that $x > 1$.
Consider $\pi = L((\rho|\cdot|^{-x})^a, \Delta_\rho[x-1,-x]^b; \pi(\phi, \eta))$ 
and $\pi_0 = L(\Delta_\rho[x-1,-x]^b; \pi(\phi, \eta))$. 
Set
\[
l = \left\{
\begin{aligned}
&\max\{a-m_{2x-1}+1,0\} 
\iif m_{2x-1}m_{2x+1} \not=0, \delta_{2x+1}\not=\kappa, \\
&\max\{a-m_{2x-1},0\} 
\other.
\end{aligned}
\right. 
\]
If $D_{\rho|\cdot|^{x}}^{(k_0)}(\pi_0)$ and $D_{\rho|\cdot|^{x}}^{(k)}(\pi)$ are the highest derivatives, 
then $k-k_0=l$. 
\end{prop}

Assume this proposition for a moment. 
If $l=0$ so that $k=k_0$, since $x \geq 1$, we have
\begin{align*}
\pi &\hookrightarrow (\rho|\cdot|^{-x})^a \rtimes \pi_0
\\&\hookrightarrow (\rho|\cdot|^x)^{k_0} \times (\rho|\cdot|^{-x})^a \rtimes D_{\rho|\cdot|^x}^{(k_0)}(\pi_0).
\end{align*}
Hence we have a non-zero map
\[
D_{\rho|\cdot|^x}^{(k)}(\pi) \rightarrow (\rho|\cdot|^{-x})^a \rtimes D_{\rho|\cdot|^x}^{(k_0)}(\pi_0).
\]
Since $D_{\rho|\cdot|^x}^{(k)}(\pi)$ is irreducible, this map must be an injection. 
If $l > 0$, then 
\[
\pi \hookrightarrow 
(\rho|\cdot|^{-x})^{l} \rtimes L((\rho|\cdot|^{-x})^{a-l}, \Delta_\rho[x-1,-x]^b; \pi(\phi, \eta)).
\]
Since $D_{\rho|\cdot|^x}^{(k_0)}(L((\rho|\cdot|^{-x})^{a-l}, \Delta_\rho[x-1,-x]^b; \pi(\phi, \eta)))$
is the highest derivative, we must have 
\begin{align*}
D_{\rho|\cdot|^x}^{(k)}(\pi) 
&= 
D_{\rho|\cdot|^x}^{(k_0)}(L((\rho|\cdot|^{-x})^{a-l}, \Delta_\rho[x-1,-x]^b; \pi(\phi, \eta)))
\\&\hookrightarrow 
(\rho|\cdot|^{-x})^{a-l} \rtimes D_{\rho|\cdot|^{x}}^{(k_0)}(\pi_0).
\end{align*}
Therefore, Proposition \ref{x>1} implies Theorem \ref{main} (for $x > 1$).
\par

Now we prove Proposition \ref{x>1}.
The proof uses Jantzen's strategy. 
\begin{proof}[Proof of Proposition \ref{x>1}]
We note that $x-1 > 0$.
As explained in \cite[\S 3.4]{J3}, 
if $\pi_1 = D_{\rho|\cdot|^{x-1}}^{(\alpha)}(\pi)$, $\pi_2 = D_{\rho|\cdot|^{x}}^{(\beta)}(\pi_1)$ 
and $\pi_3 = D_{\rho|\cdot|^{x-1}}^{(\gamma)}(\pi_2)$ are the highest derivatives, 
then $\Jac_{\rho|\cdot|^x}(\pi_3) = 0$ and
\[
\pi \hookrightarrow 
(\rho|\cdot|^{x})^{\max\{\beta-\alpha-\gamma,0\}} \times (\rho|\cdot|^{x-1})^{\max\{\alpha-\beta+\gamma,0\}} 
\times L(\rho|\cdot|^{x-1},\rho|\cdot|^{x})^{\min\{\alpha,\beta-\gamma\}} \times \Delta_\rho[x,x-1]^{\gamma} 
\rtimes \pi_3.
\]
Hence we see that $D_{\rho|\cdot|^{x}}^{(k)}(\pi)$ is the highest derivative with
\[
k = \max\{\beta-\alpha-\gamma,0\}+\gamma. 
\]
We compute $\alpha, \beta, \gamma$ by a case-by-case consideration
using Jantzen's algorithm (\S \ref{alg}).
We have to consider the following cases separately:
\begin{enumerate}
\item
$m_{2x-1}=0$; 
\item
$\delta_{2x+1} = 0$, $\delta_{2x-1}=0$ and $m_{2x-1} > 0$; 
\item
$\delta_{2x+1} = 1$ and $\delta_{2x-1}=0$; 
\item
$\delta_{2x+1} = 0$, $\delta_{2x-1} = 1$ and $m_{2x-1} \equiv 1 \bmod 2$; 
\item
$\delta_{2x+1} = 1$, $m_{2x+1} \equiv 1 \bmod 2$, $\delta_{2x-1} = 1$ and $m_{2x-1} \equiv 1 \bmod 2$; 
\item
$\delta_{2x+1} = 1$, $m_{2x+1} \equiv 0 \bmod 2$, $\delta_{2x-1} = 1$ and $m_{2x-1} \equiv 1 \bmod 2$; 
\item
$\delta_{2x+1} = 0$, $\delta_{2x-1} = 1$ and $m_{2x-1} \equiv 0 \bmod 2$; 
\item
$\delta_{2x+1} = 1$, $\delta_{2x-1} = 1$ and $m_{2x-1} \equiv 0 \bmod 2$.
\end{enumerate}
For example, 
suppose that $\delta_{2x+1} = 1$, $\delta_{2x-1} = 1$ and $m_{2x-1} \equiv 0 \bmod 2$.
Then with 
$\phi_1 = \phi - (\rho \boxtimes S_{2x-1})^{m_{2x-1}} + (\rho \boxtimes S_{2x-3})^{m_{2x-1}-2}$, 
we see that
\[
\pi_1=
D_{\rho|\cdot|^{x-1}}^{(b+m_{2x-1}-1)}(\pi) 
= 
L\left((\rho|\cdot|^{-x})^{a}, \Delta_\rho[x-2,-x]^b, \Delta_\rho[x-2,-(x-1)]; \pi(\phi_1, \eta_1) \right)
\]
is the highest derivative, so that $\alpha = b+m_{2x-1}-1$. 
Note that 
\[
\pi_1 \hookrightarrow 
L(\rho|\cdot|^{-x}, \Delta_\rho[x-2,-(x-1)]) \times \Delta_\rho[x-2, -(x-1)]^b 
\times (\rho|\cdot|^{-x})^{a+b-1} \rtimes \pi(\phi_1,\eta_1)
\]
if $a>0$.
By \cite[Proposition 3.4 (1)]{J3}, 
with $\phi_2 = \phi_1 - (\rho \boxtimes S_{2x+1})^{m_{2x+1}} + (\rho \boxtimes S_{2x-1})^{m_{2x+1}}$,
we see that
\[
\pi_2=
D_{\rho|\cdot|^{x}}^{(a+b-1+m_{2x+1})}(\pi_1) 
= L(\rho|\cdot|^{-x}, \Delta_\rho[x-2,-(x-1)]^{b+1}; \pi(\phi_2, \eta_2))
\]
is the highest derivative, so that $\beta = a+b-1+m_{2x+1}$. 
By Corollary \ref{ab} for $a=0$, we have
\[
\gamma = k_0 
= \left\{
\begin{aligned}
&m_{2x+1} \iif \text{$b \equiv 1 \bmod 2$}, \\
&m_{2x+1}-1 \iif \text{$b \equiv 0 \bmod 2$}, 
\end{aligned}
\right. 
\]
Therefore, $k-k_0 = \max\{a-m_{2x+1}-\kappa+1,0\}$, 
where $\kappa \in \{0,1\}$ such that $\kappa \equiv b \bmod 2$.
\par

In every case, a similar argument implies that $k=k_0+l$ with $l$ being in the assertion.
\end{proof}

%\subsection{The case $x=1$}
\subsection{The case $x=1$}
In this subsection, we prove: 
\begin{prop}\label{x=1}
Proposition \ref{x>1} holds even when $x=1$.
\end{prop}

As explained after Proposition \ref{x>1}, 
this proposition also implies Theorem \ref{main} for $x=1$. 
Note that Jantzen's strategy to determine $k$ from $\alpha, \beta, \gamma$ 
cannot be applied to the case $x=1$. 
Instead of this, we use the following proposition.
\begin{prop}\label{irr1}
Let $\phi \in \Phi_\gp(G)$ and $\eta \in \widehat{\Sc_\phi}$. 
Suppose that $\rho$ is self-dual of the same type as $\phi$. 
Set 
\[
\delta_3 = \left\{
\begin{aligned}
&1	
\iif \text{$\phi \supset \rho \boxtimes (S_1+S_3)$ and $\eta(\rho \boxtimes S_{1}) \not= \eta(\rho \boxtimes S_{3})$}, \\
&0	\other.
\end{aligned}
\right. 
\]
Consider $\pi = L((\rho|\cdot|^{-1})^a; \pi(\phi, \eta))$. 
Then $\rho|\cdot|^{-1} \rtimes \pi$ is irreducible if and only if $a \geq m_\phi(\rho)-\delta_3$.
\end{prop}
\begin{proof}
Write $m_1 = m_\phi(\rho)$ and $m_3 = m_\phi(\rho \boxtimes S_3)$.
By Theorem \ref{irr0}, we may assume that $a > 0$.
When $a \leq m_1$, by Example \ref{ex.xu1}, 
we have $\pi \in \Pi_{\psi_0}$ with $\psi_0 = \phi - \rho^a + (\rho \boxtimes S_1 \boxtimes S_3)^a$.
Also, if $\delta_3 = 1$ and $a \leq m_1-1$, 
by Proposition \ref{ex.xu3}, we have $\pi \in \Pi_{\psi_1}$ with
\begin{align*}
\psi_1 = \phi 
&- \rho^a + (\rho \boxtimes S_1 \boxtimes S_3)^a
\\&- (\rho + \rho \boxtimes S_3) + \rho \boxtimes S_2 \boxtimes S_2. 
\end{align*}
In particular, when $a \leq m_1-\delta_3$, 
if $D_{\rho|\cdot|^{1}}^{(k)}(\pi) \not= 0$, then $k \leq m_3-\delta_3$ by Lemma \ref{xu}. 
Since 
\begin{align*}
\pi &\hookrightarrow (\rho|\cdot|^{-1})^{a} \rtimes \pi(\phi, \eta)
\\&\hookrightarrow 
(\rho|\cdot|^1)^{m_3-\delta_3} \times (\rho|\cdot|^{-1})^{a} 
\rtimes D_{\rho|\cdot|^1}^{(m_3-\delta_3)}(\pi(\phi, \eta)), 
\end{align*}
We have $D_{\rho|\cdot|^{1}}^{(m_3-\delta_3)}(\pi) \not= 0$. 
Since it is irreducible, we have
\[
D_{\rho|\cdot|^{1}}^{(m_3-\delta_3)}(\pi) \hookrightarrow (\rho|\cdot|^{-1})^{a} 
\rtimes D_{\rho|\cdot|^1}^{(m_3-\delta_3)}(\pi(\phi, \eta)).
\]
This implies the reducibility of $\rho|\cdot|^{-1} \rtimes \pi$ when $a < m_1-\delta_3$. 
\par

Suppose that $a = m_1 - \delta_3 > 0$.
We prove the irreducibility of $\rho|\cdot|^{-1} \rtimes \pi$ by induction on $a$.
Suppose that $\rho|\cdot|^{-1} \rtimes \pi$ is reducible and choose an irreducible quotient $\sigma$. 
Note that $D_{\rho|\cdot|^{-1}}^{(a)}(\sigma) \not=0$. 
By \cite[Proposition 4.1]{J2}, 
$\sigma = L((\rho|\cdot|^{-1})^a; \pi(\phi_1, \eta_1))$ or 
$\sigma = L((\rho|\cdot|^{-1})^a, \Delta_\rho[0,-1]; \pi(\phi_2, \eta_2))$ with
\begin{align*}
\phi_1 &= \phi - \rho + \rho \boxtimes S_3, \quad
\phi_2 = \phi - \rho^2.
\end{align*}
We consider the former case $\sigma = L((\rho|\cdot|^{-1})^a; \pi(\phi_1, \eta_1))$ 
with $\phi_1 = \phi - \rho + \rho \boxtimes S_3$. 
Note that $m_{\phi_1}(\rho) = m_1 -1 \geq \delta_3$.
By induction hypothesis, 
$\sigma = \rho|\cdot|^{-1} \rtimes L((\rho|\cdot|^{-1})^{a-1}; \pi(\phi_1, \eta_1))$ 
is an irreducible induction. 
In particular, $D_{\rho|\cdot|^1}^{(m_3+2-\delta_3)}(\sigma) \not= 0$. 
Hence $D_{\rho|\cdot|^1}^{(m_3+2-\delta_3)}(\rho|\cdot|^{-1} \rtimes \pi) \not= 0$, 
which implies that $D_{\rho|\cdot|^1}^{(m_3+1-\delta_3)}(\pi) \not= 0$.
This contradicts Lemma \ref{xu}.
Now, we consider the latter case 
$\sigma = L((\rho|\cdot|^{-1})^a, \Delta_\rho[0,-1]; \pi(\phi_2, \eta_2))$
with $\phi_2 = \phi - \rho^2$ so that $m_1 \geq 2$.
Note that $\Jac_\rho(\sigma)$ is nonzero and irreducible (up to multiplicity)
since $D_{\rho}^{(2)}(\pi) = 0$ by Lemma \ref{xu}. 
If $\sigma'$ is the unique irreducible component of $\Jac_\rho(\sigma)$, 
then
\[
\sigma' = L((\rho|\cdot|^{-1})^{a+1}; \pi(\phi_2, \eta_2)).
\]
Recall that $a+1 = m_1+1-\delta_3$. 
By induction hypothesis, 
\[
\sigma' = 
\left\{
\begin{aligned}
&(\rho|\cdot|^{-1})^3 \rtimes L((\rho|\cdot|^{-1})^{m_1-2-\delta_3}; \pi(\phi_2, \eta_2)) \iif m_1>2, \\
&(\rho|\cdot|^{-1})^{3-\delta_3} \rtimes L((\rho|\cdot|^{-1})^{m_1-2}; \pi(\phi_2, \eta_2)) \iif m_1=2
\end{aligned}
\right. 
\]
is an irreducible induction.
Hence we have $D_{\rho|\cdot|^1}^{(m_3+3-\delta_3)}(\sigma') \not= 0$. 
Therefore we have 
\begin{align*}
&D_{\rho|\cdot|^1}^{(m_3+3-\delta_3)}(\Jac_\rho(\rho|\cdot|^{-1} \rtimes \pi)) \not= 0
\\&\implies 
D_{\rho|\cdot|^1}^{(m_3+2-\delta_3)}(\rho|\cdot|^{-1} \rtimes \pi) \not= 0
\\&\implies
D_{\rho|\cdot|^1}^{(m_3+1-\delta_3)}(\pi) \not= 0, 
\end{align*}
which contradicts Lemma \ref{xu}.
\par

The case where $a > m_1-\delta_3$ follows from the case $a = m_1-\delta_3$.
This completes the proof.
\end{proof}

Now we prove Proposition \ref{x=1}
\begin{proof}[Proof of Proposition \ref{x=1}]
Recall that $\pi = L((\rho|\cdot|^{-1})^a, \Delta_\rho[0,-1]^b; \pi(\phi, \eta))$ 
and $\pi_0 = L(\Delta_\rho[0,-1]^b; \pi(\phi, \eta))$, 
and that $D_{\rho|\cdot|^1}^{(k)}(\pi)$ and $D_{\rho|\cdot|^1}^{(k_0)}(\pi_0)$ are the highest derivatives. 
Set
\[
l = \left\{
\begin{aligned}
&\max\{a-m_{1}+1,0\} 
\iif m_{1}m_{3} \not=0, \delta_{3}\not=\kappa,\\
&\max\{a-m_{1},0\} 
\other.
\end{aligned}
\right. 
\]
We will prove that $k-k_0 = l$.
\par

Note that 
\[
\pi \hookrightarrow (\rho|\cdot|^{-1})^{l} \rtimes L((\rho|\cdot|^{-1})^{a-l}, \Delta_\rho[0,-1]^b; \pi(\phi, \eta)). 
\]
By applying Example \ref{ex.xu1} and Proposition \ref{ex.xu3} to 
$L((\rho|\cdot|^{-1})^{a-l}, \Delta_\rho[0,-1]^b; \pi(\phi, \eta))$, 
we see that 
$D_{\rho|\cdot|^{1}}^{(k_0)}(L((\rho|\cdot|^{-1})^{a-l}, \Delta_\rho[0,-1]^b; \pi(\phi, \eta)))$ 
is the highest derivative. 
Hence we have $k \leq k_0+l$. 
Also, we see that if $l=0$, then $k=k_0$. 
\par

Assume that $l>0$.
Since $\pi$ is a unique irreducible subrepresentation of 
\[
(\rho|\cdot|^{-1})^a \times \Delta_{\rho}[0,-1]^b \rtimes \pi(\phi, \eta) 
\cong
\Delta_{\rho}[0,-1]^b \times (\rho|\cdot|^{-1})^a \rtimes \pi(\phi, \eta), 
\]
by Proposition \ref{irr1}, we see that 
\begin{align*}
\pi 
&\hookrightarrow 
\Delta_{\rho}[0,-1]^b \times (\rho|\cdot|^{-1})^{a-m_1+\delta_3} 
\rtimes L((\rho|\cdot|^{-1})^{m_1-\delta_3}; \pi(\phi, \eta)) 
\\&\hookrightarrow
\rho^b \times (\rho|\cdot|^{-1})^{b+a-m_1+\delta_3} 
\rtimes L((\rho|\cdot|^{-1})^{m_1-\delta_3}; \pi(\phi, \eta)) 
\\&\cong 
\rho^b \times (\rho|\cdot|^{1})^{b+a-m_1+\delta_3} 
\rtimes L((\rho|\cdot|^{-1})^{m_1-\delta_3}; \pi(\phi, \eta)) 
\\&\hookrightarrow
\rho^b \times (\rho|\cdot|^{1})^{b+a-m_1+m_3} 
\rtimes D_{\rho|\cdot|^1}^{(m_3-\delta_3)}(L((\rho|\cdot|^{-1})^{m_1-\delta_3}; \pi(\phi, \eta))). 
\end{align*}
This implies that $k \geq a-m_1+m_3$. 
Since $k_0+l = a-m_1+m_3$, we conclude that $k=k_0+l$. 
\end{proof}

%References


\begin{thebibliography}{20}

\bibitem{Ar}
{J. Arthur}, 
{\em The endoscopic classification of representations. Orthogonal and symplectic groups}. 
American Mathematical Society Colloquium Publications, {\bf61}. 
{\it American Mathematical Society, Providence, RI}, 2013. xviii+590 pp.

\bibitem{At}
{H. Atobe}, 
{\em Jacquet modules and local Langlands correspondence}. 
{\it Invent. Math}. {\bf219} (2020), no.~3, 831--871.

\bibitem{Au}
{A.-M. Aubert}, 
{\em Dualit{\'e} dans le groupe de Grothendieck de la cat{\'e}gorie 
des repr{\'e}sentations lisses de longueur finie d'un groupe r{\'e}ductif $p$-adique}. 
{\it Trans.~Amer.~Math.~Soc}. {\bf347} (1995), no.~6, 2179--2189
and {\em Erratum}. ibid. {\bf348} (1996), 4687--4690.

\bibitem{GT}
{W. T. Gan and S. Takeda}, 
{\em A proof of the Howe duality conjecture}. 
{\it J. Amer. Math. Soc}. {\bf29} (2016), no.~2, 473--493.

\bibitem{J0}
{C. Jantzen}, 
{\em Jacquet modules of p-adic general linear groups}. 
{\it Represent. Theory} {\bf11} (2007), 45--83.

\bibitem{J1}
{C. Jantzen}, 
{\em Tempered representations for classical $p$-adic groups}. 
{\it Manuscripta Math}. {\bf145} (2014), no.~3-4, 319--387.

\bibitem{J2}
{C. Jantzen}, 
{\em Jacquet modules and irrreducibility of induced representations for classical $p$-adic groups}. 
{\it Manuscripta Math}. {\bf156} (2018), no.~1-2, 23--55.

\bibitem{J3}
{C. Jantzen}, 
{\em Duality for classical $p$-adic groups: the half-integral case}. 
{\it Represent. Theory} {\bf22} (2018), 160--201.

\bibitem{K}
{T. Konno}, 
{\em A note on the Langlands classification and irreducibility of induced representations of $p$-adic groups}. 
{\it Kyushu J. Math}. {\bf57} (2003), no.~2, 383--409. 

\bibitem{KL}
{A. Kret and E. Lapid}, 
{\em Jacquet modules of ladder representations}. 
{\it C. R. Math. Acad. Sci. Paris} {\bf350} (2012), no.~21-22, 937--940.

\bibitem{LM}
{E. Lapid and A. M{\'i}nguez}, 
{\em On a determinantal formula of Tadi{\'c}}. 
{\it Amer. J. Math}. {\bf136} (2014), no.~1, 111--142. 

\bibitem{LM1}
{E. Lapid and A. M{\'i}nguez}, 
{\em On parabolic induction on inner forms of the general linear group over a non-archimedean local field}.
{\it Sel. Math. New Ser}. {\bf22}, 2347--2400 (2016). 

\bibitem{LT}
{E. Lapid and M. Tadi{\'c}}, 
{\em Some results on reducibility of parabolic induction for classical groups}. 
Preprint 2017. 
Available at: 
\url{http://www.hazu.hr/~tadic/58-Lapid-T-ladder-red.pdf}.

\bibitem{M}
{A. M{\'i}nguez}, 
{\em Correspondance de Howe explicite: paires duales de type II}. 
{\it Ann. Sci. {\'E}c. Norm. Sup{\'e}r. (4)} {\bf41} (2008), no.~5, 717--741. 


\bibitem{M1}
{A. M{\'i}nguez}, 
{\em Sur l'irr{\'e}ductibilit{\'e} d'une induite parabolique}. 
{\it J. Reine Angew. Math}. {\bf629} (2009), 107--131. 

\bibitem{Moe1}
{C. M{\oe}glin}, 
{\em Sur certains paquets d'Arthur et involution d'Aubert--Schneider--Stuhler g{\'e}n{\'e}ralis{\'e}e}. 
{\it Represent.~Theory} {\bf10} (2006), 86--129.

\bibitem{Moe2}
{C. M{\oe}glin}, 
{\em Paquets d'Arthur discrets pour un groupe classique $p$-adique}. 
{\it Automorphic forms and $L$-functions $\mathrm{II}$. Local aspects}, 179--257, 
Contemp.~Math., {\bf489}, Israel Math.~Conf.~Proc., 
{\it Amer.~Math.~Soc., Providence, RI}, 2009.

\bibitem{MoeC}
{C. M{\oe}glin}, 
{\em Comparaison des param{\`e}tres de Langlands 
et des exposants {\`a} l'int{\'e}rieur d'un paquet d'Arthur}. 
{\it J.~Lie Theory} {\bf19} (2009), no.~4, 797--840.

\bibitem{Moe3}
{C. M{\oe}glin}, 
{\em Multiplicit{\'e} $1$ dans les paquets d'Arthur aux places $p$-adiques}. 
{\it On certain $L$-functions}, 333--374, 
Clay Math.~Proc., {\bf13}, {\it Amer.~Math.~Soc., Providence, RI}, 2011.

\bibitem{MW}
{C. M{\oe}glin and J.-L. Waldspurger}, 
{\em Le spectre r{\'e}siduel de $\GL(n)$}. 
{\it Ann.~Sci.~{\'E}cole Norm.~Sup.~(4)} {\bf22} (1989), no.~4, 605--674.

\bibitem{T}
{M. Tadi{\'c}}, 
{\em Structure arising from induction and Jacquet modules of representations of classical $p$-adic groups}. 
{\it J. Algebra} {\bf177} (1995), no.~1, 1--33.

\bibitem{X1}
{B. Xu}, 
{\em On the cuspidal support of discrete series for $p$-adic quasisplit $Sp(N)$ and $SO(N)$}. 
{\it Manuscripta Math}. {\bf154} (2017), no.~3-4, 441--502. 

\bibitem{X2}
{B. Xu}, 
{\em On M{\oe}glin's parametrization of Arthur packets for $p$-adic quasisplit $\Sp(N)$ and $\SO(N)$}. 
{\it Canad.~J. Math}. {\bf69} (2017), no.~4, 890--960.

\bibitem{X3}
{B. Xu}, 
{\em A combinatorial solution to M{\oe}glin's parametrization of Arthur packets 
for $p$-adic quasisplit $\Sp(N)$ and $\mathrm{O}(N)$}. 
{\it J. Inst.~Math.~Jussieu} (2019). 
\url{http://dx.doi.org/10.1017/S1474748019000409}.

\bibitem{Z}
{A. V. Zelevinsky}, 
{\em Induced representations of reductive $\mathfrak{p}$-adic groups. II. 
On irreducible representations of $\GL(n)$}.  
{\it Ann. Sci. {\'E}cole Norm. Sup}. (4) {\bf13} (1980), no.~2, 165--210.

\end{thebibliography}
\end{document}